\documentclass[preprint]{imsart}

\RequirePackage[OT1]{fontenc}
\RequirePackage{amsthm,amsmath,latexsym}
\RequirePackage[numbers]{natbib}
\RequirePackage[colorlinks,citecolor=blue,urlcolor=blue]{hyperref}

\usepackage[dvips]{graphicx}

\usepackage{amsfonts}

\startlocaldefs
\numberwithin{equation}{section}
\theoremstyle{plain}

\newtheorem{definition}{Definition}[section]
\newtheorem{proposition}[definition]{Proposition}
\newtheorem{theorem}[definition]{Theorem}
\newtheorem{corollary}[definition]{Corollary}
\newtheorem{lemma}[definition]{Lemma}

\newtheorem{example}{Example}
\newtheorem{remark}{Remark}

\endlocaldefs

\def\Eta{E}

\def\argmax{\mathop{\rm argmax}}
\def\argmin{\mathop{\rm argmin}}

\newcommand{\cX}{{\cal X}}
\newcommand{\cY}{{\cal Y}}

\def\cY{{\cal Y}}

\makeatletter
\newcommand{\lleq}{\mathrel{\mathpalette\gl@align<}}
\newcommand{\ggeq}{\mathrel{\mathpalette\gl@align>}}
\newcommand{\gl@align}[2]{
\vbox{\baselineskip\z@skip\lineskip\z@
\ialign{$\m@th#1\hfil##\hfil$\crcr#2\crcr{}_{{}_{(=)}}\crcr}}}
\makeatother

\def\Label#1{\label{#1}\ [\ \text{#1}\ ]\ }
\def\Label{\label}

\newenvironment{proofof}[1]{\vspace*{5mm} \par \noindent
         \quad{\it Proof of #1:\hspace{2mm}}}{\endproof
\hfill$\Box$ 
}

\begin{document}

\begin{frontmatter}
\title{Information Geometry Approach to Parameter Estimation in Markov Chains}
\runtitle{Information Geometry Approach in Markov Chains}

\begin{aug}

\author{{Masahito Hayashi$^{1}$ 
and Shun Watanabe$^{2}$}\\
\rm $^1$Graduate School of Mathematics, Nagoya University, Japan, \\
and Centre for Quantum Technologies, National University of Singapore, Singapore. \\
$^2$Department of  Information Science and Intelligent Systems, 
University of Tokushima, Japan, \\
and Institute for System Research, University of Maryland, College Park. \\
}
\runauthor{M. Hayashi and S. Watanabe}

\end{aug}

\begin{abstract}
We consider the parameter estimation of Markov chain 
when the unknown transition matrix belongs to
an exponential family of transition matrices. 
Then, we show that the sample mean of the generator of 
the exponential family
is an asymptotically efficient estimator.
Further, we also define a curved exponential family of transition matrices.
Using a transition matrix version of the Pythagorean theorem,
we give an asymptotically efficient estimator for a curved exponential family.
\end{abstract}

\begin{keyword}[class=MSC]
\kwd[Primary ]{62M05}
\end{keyword}

\begin{keyword}
\kwd{exponential family}
\kwd{natural parameter}
\kwd{expectation parameter} 
\kwd{relative entropy}
\kwd{Fisher information matrix}
\kwd{asymptotic efficient estimator}
\end{keyword}

\end{frontmatter}

\section{Introduction}\Label{s1}
Information geometry established by Amari and Nagaoka \cite{AN}
is an elegant method for statistical inference.
This method provides us a very general approach to statistical parameter estimation. 
Under this framework,
we easily find that the efficient estimator 
can be given with less calculation complexity
for exponential families and a curved exponential families
under the independent and identical distributed case.
Therefore, we can expect a similar structure in the Markov chains.

The preceding studies
\cite{Feigin,KS,Hudson,Bhat,Bhat2,Stefanov,Kuchler-Sorensen,Sorensen}
introduced the concept of exponential families of transition matrices.
However, in their definition,  
although the maximum likelihood estimator has the asymptotic efficiency,
i.e., attains the Cram\'{e}r-Rao bound asymptotically,
the maximum likelihood estimator is not necessarily calculated with less calculation complexity.
That is, the maximum likelihood estimator has a complex form so that
it requires long calculation time in their model.
Further, 
it is quite difficult to calculate the Cram\'{e}r-Rao bound 
even with the asymptotic first order coefficient
because these papers focused only on the limit of the inverse of the Fisher information.
From a practical viewpoint, it is needed to calculate the asymptotic first order coefficient.
So, it is strongly required to resolve these two problems for the estimation of Markovian process, i.e.,
(1) to give an asymptotically efficient estimator with small calculation
and  (2) to derive a formula for the asymptotic Cram\'{e}r-Rao bound 
with small calculation.

The purpose of this paper is giving the answers for these two problems.
For this purpose, 
we notice another type of exponential family of transition matrices by 
Nakagawa and Kanaya \cite{NK} and Nagaoka \cite{HN}. 
They defined the Fisher information matrix in their sense.
On the other hand, 
for the estimation of the probability distribution,
the class of curved exponential families plays an important role as a wider class of
distribution families than the class of exponential families.
That is, when the unknown distribution belongs to a curved exponential family,
the asymptotic efficient estimator can be treated in the information-geometrical framework.
Therefore, to deal with these problems in a wider class of families of transition matrices,
we introduce a curved exponential family of transition matrices as a subset of an exponential family of transition matrices in the sense of \cite{NK,HN}.
Since any exponential family of transition matrices is a curved exponential family,
the class of curved exponential families is a larger class of
families of transition matrices
than the class of exponential families.
Especially, any smooth subset of transition matrices on a finite-size system forms a curved exponential family of transition matrices.
Our purpose is resolving the above two problems 
for a curved exponential family as well as for an exponential family.
Since any smooth parametric subfamily of transition matrices on a finite-size system forms a curved exponential family,
our treatment for curved exponential families
has a wide applicability for the estimation of Markovian process.
This is reason why we adopted the definition of an exponential family
by  \cite{NK,HN}. 

Firstly, we  show that, for an exponential family of transition matrices in the sense of \cite{NK,HN},
an estimator of a simple form asymptotically attains the Cram\'{e}r-Rao bound, which is given as the inverse of Fisher information matrix.
That is, the estimator for the expectation parameter
is asymptotically efficient and is written as the sample mean of 
$n+1$-observations.
Since it requires only a small amount of calculation, 
the problem (1) is resolved.
Additionally, the problem (2) is also resolved for an exponential family of transition matrices because Fisher information matrix is computable.

To show the above items,
we discuss the behavior of the sample mean of $n+1$ observations.
Indeed, while the existing papers \cite{kemeny-snell-book,Peskun} derived the form of the asymptotic variance,
this paper shows that the asymptotic variance can be written by using the second derivative of 
the potential function of the generated exponential family.
Using this relation, we show that the sample mean 
asymptotically attains the Cram\'{e}r-Rao bound
for the expectation parameter.

Next, we define the Fisher information matrix for a curved exponential family with a computable form.
Then, using a transition matrix version of the Pythagorean theorem,
we give an asymptotically efficient estimator for a curved exponential family, in which, the estimator is given as a function of the above estimator in the larger exponential family.
Since the asymptotic mean square error is the inverse of the Fisher information matrix,
the problems (1) and (2) are resolved jointly. 
In the above way, we resolve the problems that were unsolved in existing papers
\cite{Feigin,KS,Hudson,Bhat,Bhat2,Stefanov,Kuchler-Sorensen,Sorensen}.
Further, during this derivation, we also obtain a notable evaluation 
for variance of sample mean as a by product, which is summarized in Subsection \ref{s2-1}.

For the above discussion, we need the description of 
an exponential family of transition matrices.
Since the information geometrical structure for probability distributions 
plays important roles in several topics in information theory as well as statistics,
it is better to describe the information geometry of transition matrices 
so that it can be easily applied to these topics.
In fact, the authors applied it to finite-length evaluations of the tail probability,
the error probability in simple hypothesis testing, 
source coding, channel coding, and random number generation in Markov chain as well as the estimation error of parametric family of transition matrices \cite{HW14-2,W-H2}.
Thus, we revisit the exponential family of transition matrices \cite{NK,HN}
in a manner consistent with the above purpose
by using Bregmann divergence \cite{Br,Am}.
In particular, the relative R\'{e}nyi entropy for transition matrices plays an important role in the finite-length analysis; we define the relative entropy for transition matrices so that it is a special case of the relative R\'{e}nyi entropy, which is different from the definitions in the literatures \cite{NK,HN}. 
Although some of results in this paper have been already stated
in \cite{HN} (without detailed proof), we restate those results and give proofs since the logical order of arguments are
different from \cite{HN} and we want to keep the paper self-contained.
In particular, although the paper \cite{HN} is written with differential geometrical terminologies, e.g., Christoffel symbols,
this paper is written only with terminologies of convex functions and linear algebra. 

The remaining of this paper is organized as follows.
Section \ref{s2} gives the brief summary of obtained results, which is crucial for understanding the structure of this paper.
In Section \ref{s4}, we define 
the relative entropy and the relative R\'{e}nyi entropy between two transition matrices
In Section \ref{s4-5}, we revisit an exponential family of transition matrices
and its properties.
In Section \ref{s5}, we focus on the joint distribution when a transition matrix is given 
as an element of a one-parameter exponential family
and the input distribution is given as the stationary distribution.
Then, we characterize the quantities given in Sections \ref{s4} and \ref{s4-5} by using the joint distribution.
In Section \ref{s6}, we proceed to the $n+1$ observation Markov process when the initial distribution is the stationary distribution.
Then, we show that the sample mean of the generator is an unbiased and asymptotically efficient estimator 
under a one-parameter exponential family.
In Section \ref{s7}, we proceed to the $n+1$ observation Markov process when the initial distribution is a non-stationary distribution.
We show a similar fact in this case.
Section \ref{s8} extends a part of these results 
to the multi-parameter case and the case of a curved exponential family. 
In appendix, we address the relations with existing results by Nakagawa and Kanaya \cite{HN}, Nagaoka\cite{HN}, and Natarajan \cite{N}.

\section{Summary of results}\Label{s2}
Here, we prepare notations and definitions.
For two given transition matrices $W$ and $W_Y$ over $\cX$ and $\cY$,
we define $W \times W_Y(x,y|x',y'):=W(x|x')W_Y(y|y')$,
$W^{\times n}(x_n,x_{n-1}, \ldots,x_1|x'):=
W(x_n|x_{n-1})W(x_{n-1}|x_{n-2})\cdots W(x_1|x')$,
and 
$W^n(x|x')= \sum_{x_{n-1}, \ldots,x_1} W^{\times n}(x,x_{n-1}, \ldots,x_1|x')$.
For a given distribution $P$ on $\cX$
and a transition matrix $V$ from $\cX$ to $\cY$,
we define $V \times P(y,x):=V(y|x)P(x)$
and $V P(y):=\sum_x V \times P(y,x)$.

A non-negative matrix $W$ is called {\it irreducible} 
when for each $x,x'\in \cX$, there exists a natural number $n$ such that $W^n(x|x')>0$ \cite{MU}.
An irreducible matrix $W$ is called {\it ergodic} when 
there are no input $x'$ and no integer $n'$
such that $W^{n} (x'|x')=0$ unless 
$n$ is divisible by $n'$ \cite{MU}.
The irreducibility and the ergodicity depend only on the support $\cX^2_W:=\{(x,x') \in \cX^2| W(x|x')>0\}$
for a non-negative matrix $W$ over $\cX$.
Hence, we say that $\cX^2_W$ is {\it irreducible} and {\it ergodic} 
when a non-negative matrix $W$ is irreducible and ergodic, respectively.
Indeed, when a subset of $\cX^2_W$ is irreducible and ergodic,
the set $\cX^2_W$ is also irreducible and ergodic, respectively.
It is known that 
the output distribution $W^n P$ converges to the stationary distribution of $W$
for a given ergodic transition matrix $W$ \cite{kemeny-snell-book,DZ,MU}.
Although the main result is asymptotic estimation for 
an exponential family and a curved exponential family,
we also have additional results as Subsections \ref{s2-1} and \ref{s2-3}.

\subsection{Asymptotic behavior of sample mean}\Label{s2-1}
Assume that the random variable $X_n$ obeys the Markov process with the 
irreducible and ergodic transition matrix $W(x|x')$.
In this paper, for an arbitrary two-input function $g(x,x')$,
we focus on the sample mean $S_n:=\frac{g^n(X^{n+1})}{n}$
where
$g^n(X^{n+1}):= \sum_{i=1}^n g(X_{i+1},X_{i})$,
and $X^{n+1}:= (X_{n+1}, \ldots,X_1)$.
This is because a two-input function $g(x,x')$ is closely related to an exponential family of transition matrices.
Indeed, the simple sample mean can be treated in this formulation by choosing $g(x,x')$ as $x$ or $x'$.
Since the function $g(x,x')$ can be chosen arbitrary, 
the following discussion can handle the sample mean of the hidden Markov process.

Then, the expectation $\mathsf{E} [S_n]$ and the variance $\mathsf{V} [S_n]$
are characterized as follows.
We denote the normalized Perron-Frobenius eigenvector of 
$W(x|x')$ by $P_W$ and
define the limiting expectation $
\mathsf{E}[g(X,X')]:= \sum_{x,x'}g(x,x')W(x|x') P_W(x')$.
We denote the Perron-Frobenius eigenvalue of 
$W(x|x')e^{\theta g(x,x')}$ by $\lambda_\theta$
and define the cumulant generating function $\phi(\theta):= 
\log \lambda_\theta$.
Then, when the transition matrix $W$ is irreducible and ergodic,
the relation
\begin{align}
\mathsf{E} [S_n] &\to \mathsf{E}[g(X,X')] 
\end{align}
is known.
In Sections \ref{s6} and \ref{s7} of this paper, we show
\begin{align}
\Label{1-15-1} n \mathsf{V} [S_n] & \to \frac{d^2 \phi}{d \theta^2}(0)
\end{align}
while existing papers \cite{kemeny-snell-book,Peskun} characterized the asymptotic variance 
by using the fundamental matrix. (See \cite[Section 6]{HW14-2}.)

In particular, when the initial distribution is the stationary distribution $P_W$, we have
$\mathsf{E} [S_n]=  \mathsf{E}[g(X,X')]$.
Then, in Section \ref{s6}, using a constant $C$, we show that
\begin{align}
\frac{d^2 \phi}{d \theta^2}(0) (1-\frac{C}{\sqrt{n}})^2
\le n \mathsf{V} [S_n] \le  
\frac{d^2 \phi}{d \theta^2}(0) (1+\frac{C}{\sqrt{n}})^2
\end{align}
for the stationary case. The concrete form of $C$ is also  given in Section \ref{s6}.
This analysis is obtained via evaluations of Fisher information given in Sections \ref{s5}, \ref{s6}, and \ref{s7}.

\subsection{Cram\'{e}r-Rao bound and asymptotically efficient estimator}\Label{s2-3}
Firstly, for simplicity, we summarize our obtained results for the one-parameter case
while this paper addresses a multi-parameter exponential family.
In Section \ref{s4-5}, for a given two-input function $g(x,x')$ and an irreducible and ergodic 
transition matrix $W$,
we define the potential function $\phi(\theta)$ and
exponential family of transition matrices
$\{W_\theta\}$ with the generator $g(x,x')$.
We also define its Fisher information matrix $\frac{d^2 \phi}{d\theta^2}(\theta)$
and the expectation parameter 
$\eta(\theta):=\frac{d \phi}{d \theta}(\theta)$. 
Then, we focus on the distribution family of Markov chains generated by the family of transition matrices $\{ W_\theta \}$
with arbitrary initial distributions.
We show that 
the Fisher information of the expectation parameter under the distribution family
is asymptotically equal to 
$n \frac{d^2 \phi}{d\theta^2}(\theta(\eta))^{-1}+o(n) $ 
even for the non-stationary case in Section \ref{s7}.
Then, we show that the random variable $S_n$
is the asymptotically efficient estimator, i.e.,
the mean square error is $\frac{d^2 \phi}{d\theta^2}(\theta(\eta))/n+o(1/n)$. 
In Section \ref{s6}, we give more detailed analysis for the stationary case.
To derive the results in Sections \ref{s6} and \ref{s7}, 
we prepare evaluations of Fisher information in Section \ref{s5}.

Now, we address the multi-parameter case.
In Section \ref{s4-5},
we also define a multi-parameter exponential family $W_{\vec{\theta}}$ of transition matrices,
and show the Pythagorean theorem.
Then, we show the asymptotic efficiency of the sample mean in the multi-parameter case in Subsections \ref{s8-1} and \ref{s8-2}.
We also show that 
the set of all positive transition matrices on
a finite-size system forms an exponential family in Example \ref{1-3-8}.
Further, we define a curved exponential family of transition matrices,
and give its asymptotically efficient estimator in Subsection \ref{s8-3}.
Since any smooth parametric family of transition matrices on a finite-size system forms a curved exponential family,
this result has a wide applicability.
These results require the technical preparations given in Sections \ref{s4}, \ref{s4-5}, and \ref{s5}.

\subsection{Relative entropy and relative R\'{e}nyi entropy}
In this paper, given two transition matrices $W$ and $V$,
we define the relative entropy $D(W\|V)$ and the relative R\'{e}nyi entropy $D_{1+s}(W\|V)$ in Section \ref{s4}.
In Subsection \ref{s8-3}, 
the relative entropy $D(W\|V)$ plays a crucial role 
in our estimator in a curved exponential family.
We also show that 
the Fisher information is given as the limits of 
the relative entropy and the relative R\'{e}nyi entropy,
which plays important roles in the proof of the asymptotic efficiency of our estimator in a curved exponential family in Subsection \ref{s8-3}.
Also, as discussed in \cite{HW14-2},
the relative R\'{e}nyi entropy $D_{1+s}(W\|V)$
plays a central role in simple hypothesis testing 
as well as the relative entropy $D(W\|V)$.
Further, these information quantities play an central role in 
random number generation, data compression,
and channel coding \cite{W-H2}. 
In Section \ref{s4}, we also give their properties that are useful in the above applications.

For these applications, we need to address
the relative entropy $D(W\|V)$ and the relative R\'{e}nyi entropy $D_{1+s}(W\|V)$ in a unified way.
More precisely, the relative entropy $D(W\|V)$ is needed to be defined as the limit of the relative R\'{e}nyi entropy $D_{1+s}(W\|V)$.
Indeed, the existing paper \cite{HN} defined
the relative entropy $D(W\|V)$ in a different way.
However, the definition by \cite{HN} cannot yield the definition of 
the relative R\'{e}nyi entropy in a unified way.
Appendix \ref{as1} summarizes the detailed relation between the results in this part and existing results.


\section{Relative entropy and relative R\'{e}nyi entropy}\Label{s4}
In this section, in order to investigate geometric structure for transition matrices,
we define the relative entropy and the relative R\'{e}nyi entropy.
For this purpose we prepare the following lemma, which is shown after Lemma \ref{L11}.
\begin{lemma}\Label{L1}
Consider an irreducible transition matrix $W$ over $\cX$
and a real-valued function $g$ on $\cX \times \cX$.
Define $\phi(\theta)$ as the logarithm of the Perron-Frobenius eigenvalue of the matrix:
\begin{align}
\overline{W}_{\theta}(x|x'):= W(x|x') e^{\theta g(x,x')}.
\end{align}
Then, the function $\phi(\theta)$ is convex.
Further, the following conditions are equivalent.
\begin{itemize}
\item[(1)] 
No real-valued function $f$ on $\cX$ satisfies that 
$g(x,x')=f(x)-f(x')+c$ for any $(x,x')\in \cX^2_W$ with a constant $c \in \mathbb{R}$.
\item[(2)]
The function $\phi(\theta)$ is strictly convex, i.e., 
$\frac{d^2 \phi}{d \theta^2}(\theta)>0$ for any $\theta$.
\item[(3)]
$\frac{d^2 \phi}{d \theta^2}(\theta)|_{\theta=0}>0$.
\end{itemize}
\end{lemma}

Using Lemma \ref{L1}, given 
two distinct transition matrices $W$ and $V$,
we define the relative entropy $D({W} \| V)$
and the relative R\'{e}nyi entropy $D_{1+s}({W} \| V)$ as follows. 
For this purpose, we denote the logarithm of 
the Perron-Frobenius eigenvalue of the matrix
$W(x|x')^{1+s}V(x|x')^{-s}$ by $\varphi(1+s)$ under the condition given below.
When $\cX^2_W \subset \cX^2_V$ and $\cX^2_W$ is irreducible,
we define 
\begin{align}
D({W} \| V):= 
\frac{d \varphi}{d s}(1) \Label{1-10},\quad
D_{1+s}({W} \| V):=
\frac{\varphi(1+s)}{s}
\end{align}
for $s>0$.
The relative R\'{e}nyi entropy $D_{1+s}({W} \| V)$ with $s \in (-1,0)$
is defined by (\ref{1-10}) when $\cX^2_W \cap \cX^2_V$ is irreducible,
which is a weaker assumption.
When $\cX^2_W \cap \cX^2_V$ is irreducible and 
the condition $\cX^2_W \subset \cX^2_V$ does not hold,
the relative entropy $D({W} \| V)$
and the relative R\'{e}nyi entropy $D_{1+s}({W} \| V)$
with $s>0$ are regarded as the infinity.
Note that the limit $\lim_{s \to 0}D_{1+s}({W} \| {W}')$ equals
$D({W} \| {W}')$.
When $\cX^2_W \subset \cX^2_V$ and $\cX^2_W$ is irreducible,
the function $\log \frac{W(x|x')}{V(x|x')}$ satisfies 
the condition for the function $g$ in Lemma \ref{L1}
because $W$ and $V$ are distinct.
Hence, the function $s \mapsto s D_{1+s}({W} \| V)$ is strictly convex.
So, the relative R\'{e}nyi entropy $D_{1+s}({W} \| V)$ is strictly monotone increasing with respect to $s$.

From the property of Perron-Frobenius eigenvalue,
we immediately obtain the following lemma.
\begin{lemma}\Label{L-1-11-3}
Given two transition matrices $W_X$ and $V_X$ 
($W_Y$ and $V_Y$) on ${\cal X}$ (${\cal Y}$), respectively,
we have
\begin{align*}
D(W_X\|V_X) + D(W_Y\|V_Y) &=
D(W_X\times W_Y\|V_X\times V_Y) \\
D_{1+s}(W_X\|V_X) + D_{1+s}(W_Y\|V_Y) &=
D_{1+s}(W_X\times W_Y\|V_X\times V_Y) 
\end{align*}
for $s \in (-1,0)\cup (0,\infty)$.
\end{lemma}

\begin{theorem}\Label{th15}
Transition matrices $W_1$, $W_2$, and $W$ satisfy
\begin{align}
p D(W_1\|W)+ (1-p) D(W_2\|W) & \ge D(p W_1+ (1-p) W_2 \|W)
\Label{th15-1} \\
p D(W\| W_1)+ (1-p) D(W\| W_2) & \ge D(W \| p W_1+ (1-p) W_2 )
\Label{th15-2} 
\end{align}
for $p \in (0,1)$.
\end{theorem}

\eqref{th15-1} 
can be directly shown from Lemma \ref{L1-14} 
given latter.
The proof of \eqref{th15-2}
will be given after \eqref{eq:markov-divergence-cdf-relation}.

\section{Information geometry for transition matrices}\Label{s4-5}
\subsection{Exponential family}
In the following, we treat only irreducible transition matrices. Hence, 
an irreducible transition matrix is simply called a transition matrix.
We define an exponential family for transition matrices.
We focus on a transition matrix $W(x|x')$ from $\cX$ to $\cX$.
Then, a set of real-valued functions $\{g_j\}$ on $\cX \times \cX$
is called {\it linearly independent} under the transition matrix $W(x|x')$
when any linear non-zero combination of $\{g_j\}$ satisfies the condition in Lemma \ref{L1}.
For $\vec{\theta}=(\theta^1, \ldots,\theta^d)$
and linearly independent functions $\{g_j\}$,
we define the matrix ${W}_{\vec{\theta}}(x|x')$ from $\cX$ to $\cX$ in the following way.
\begin{align}
\overline{W}_{\vec{\theta}}(x|x'):= W(x|x') e^{\sum_{j=1}^d  \theta^j g_j(x,x')}.
\Label{5-6}
\end{align}
Using the Perron-Frobenius eigenvalue $\lambda_{\vec{\theta}}$ of $\overline{W}_{\vec{\theta}}$,
we define the potential function $\phi(\vec{\theta}):=\log \lambda_{\vec{\theta}}$. 

Note that, since the value $\sum_{x}\overline{W}_{\vec{\theta}}(x|x')$ generally depends on $x'$, 
we cannot make a transition matrix by simply multiplying a constant with the matrix $\overline{W}_{\vec{\theta}}$.
To make a transition matrix from the matrix $\overline{W}_{\vec{\theta}}$, we recall that
a non-negative matrix $V$ from $\cX$ to $\cX$ is a transition matrix
if and only if the vector $(1, \ldots,1)^T$ is an eigenvector of the transpose $V^T$.
In order to resolve this problem, we focus on the structure of the matrix $\overline{W}_{\vec{\theta}}$.
We denote the Perron-Frobenius eigenvectors
of $\overline{W}_{\vec{\theta}}$ and its transpose $\overline{W}_{\vec{\theta}}^T$
by $\overline{P}^2_{\vec{\theta}}$ and $\overline{P}^3_{\vec{\theta}}$.
Then, 
similar to \cite[(16)]{NK} \cite[(2)]{HN},
we define the matrix ${W}_{\vec{\theta}}(x|x')$ as
\begin{align}
{W}_{\vec{\theta}}(x|x'):= \lambda_{\vec{\theta}}^{-1} \overline{P}^3_{\vec{\theta}}(x)
\overline{W}_{\vec{\theta}}(x|x')\overline{P}^3_{\vec{\theta}}(x')^{-1}.\Label{12-26-1}
\end{align}
The matrix ${W}_{\vec{\theta}}(x|x')$ is 
a transition matrix because the vector $(1, \ldots,1)^T$ is an eigenvector of the transpose $W_{\vec{\theta}}^T$.
The stationary distribution of the given 
transition matrix ${W}_{\vec{\theta}}$
is the Perron-Frobenius normalized eigenvector of the transition matrix ${W}_{\vec{\theta}}$,
which is given as
\begin{align}
\overline{P}^1_{\vec{\theta}}(x):=
\frac{\overline{P}^3_{\vec{\theta}}(x)\overline{P}^2_{\vec{\theta}}(x)}{
\sum_{x''}\overline{P}^3_{\vec{\theta}}(x'')\overline{P}^2_{\vec{\theta}}(x'')}
\end{align}
because 
\begin{align*}
& \sum_{x'}W_{\vec{\theta}}(x|x') \overline{P}^1_{\vec{\theta}}(x')
=
\frac{\overline{P}^3_{\vec{\theta}}(x)}{
\lambda_{\vec{\theta}} \sum_{x''}\overline{P}^3_{\vec{\theta}}(x'')
\overline{P}^2_{\vec{\theta}}(x'')}
\sum_{x'}\overline{W}_{\vec{\theta}}(x|x') 
\overline{P}^2_{\vec{\theta}}(x') \nonumber \\
=&
\frac{\overline{P}^3_{\vec{\theta}}(x) \overline{P}^2_{\vec{\theta}}(x)}{
\sum_{x''}\overline{P}^3_{\vec{\theta}}(x'')\overline{P}^2_{\vec{\theta}}(x'')}
=
\overline{P}^1_{\vec{\theta}}(x).
\end{align*}
In the following, we call the family of transition matrices 
${\cal E}:=\{ {W}_{\vec{\theta}} \}$ an {\it exponential family} of transition matrices 
generated by $W$ with the generator $\{g_1,\ldots,g_d\}$.

Since the generator $\{g_1,\ldots,g_d\}$ is linearly independent,
due to Lemma \ref{L1},
$\sum_{i,j}c_i c_j\frac{\partial^2 \phi}{\partial \theta^i\partial \theta^j}=
\frac{d^2 \phi(\vec{c}t)}{d t^2}$ 
is strictly positive for an arbitrary non-zero vector $\vec{c}=(c_1, \ldots, c_d)$.
That is, the Hesse matrix 
$\mathsf{H}_{\vec{\theta}} [\phi]=[\frac{\partial^2 \phi}{\partial \theta^i\partial \theta^j}]_{i,j}$ 
is non-negative.

Using the potential function $\phi(\theta)$,
we discuss several concepts 
for transition matrices based on Lemma \ref{L1}, formally.
We call the parameter $(\theta^1,\ldots,\theta^d)$ the {\it natural parameter}, and the parameter $\eta_j(\vec{\theta}):= 
\frac{\partial \phi}{\partial\theta^j}(\vec{\theta})$ the {\it expectation parameter}.
For $\vec{\eta}=(\eta_1, \ldots,\eta_d)$,
we define $\theta^1(\vec{\eta}), \ldots, \theta^d(\vec{\eta})$
as $\eta_j(\theta^1(\vec{\eta}), \ldots, \theta^d(\vec{\eta}))=\eta_j$.

For a given transition matrix $W$,
we define 
a linear subspace ${\cal N}_W({\cal X}^2)$ of
the space ${\cal G}({\cal X}^2)$ of all two-input functions
as the set of functions $f(x)-f(x')+c$.
Then, we obtain the following lemma.
\begin{lemma}\Label{L1-14-2}
The following are equivalent for the generator $\{g_j\}$
and the transition matrix $W$.
\begin{itemize}
\item[(1)]
The set of functions $\{g_j\}$ are linearly independent in the 
quotient space ${\cal G}({\cal X}^2)/{\cal N}_W({\cal X}^2)$.

\item[(2)]
The map $\vec{\theta} \to \vec{\eta}(\vec{\theta})$
is one-to-one.

\item[(3)]
The Hesse matrix $\mathsf{H}_{\vec{\theta}} [\phi]$ is strictly positive for any $\vec{\theta}$,
which implies the strict convexity of the potential function $\phi(\vec{\theta})$.

\item[(4)]
The Hesse matrix $\mathsf{H}_{\vec{\theta}} [\phi]|_{\vec{\theta}=0}$ is strictly positive.

\item[(5)]
The parametrization 
$\vec{\theta}\mapsto W_{\vec{\theta}}$ is faithful for any $\vec{\theta}$.
\end{itemize}
\end{lemma}

\begin{proof}
Applying Lemma \ref{L1} to $\phi(\vec{c}t)$
for an arbitrary non-zero vector $\vec{c}=(c_1, \ldots, c_d
)$, we obtain 
the equivalence among (1), (3), and (4).
(3) $\Rightarrow$ (2) is trivial.

Now, we show (2) $\Rightarrow$ (1) by showing the contraposition.
If (1) does not holds.
There exists a non-zero vector $\vec{c}=(c_1, \ldots, c_d
)$ such that 
$\sum_i c_i g_i(x,x')= f(x)-f(x')+C$.
Hence,
we have $\frac{d^2 \phi(\vec{c}t)}{d t^2}=0$.
Hence, (2) does not hold.

Now, we show (1) $\Rightarrow$ (5) by showing the contraposition.
When $W_{\vec{\theta'}}=W_{\vec{\theta}}$,
considering the logarithm,
there exist a function $f$ and a constant $c$ such that
$\sum_{j}{\theta'}^j g_j(x,x') -\sum_{j}{\theta}^j g_j(x,x') = f(x)-f(x')+C$
for $(x,x')\in \cX^2_W$.

Now, we show (5) $\Rightarrow$ (1) by showing the contraposition.
If a set of real-valued functions $\{g_j\}$ on $\cX \times \cX$ is not linearly independent,
there exist a function $f$ and a constant $C$ such that
$\sum_{j}{\theta'}^j g_j(x,x') -\sum_{j}{\theta}^j g_j(x,x')
= f(x)-f(x')+C$.
In this case, 
choosing 
$\overline{P}^3_{\vec{\theta'}}(x)=
\overline{P}^3_{\vec{\theta}}(x) e^{f(x)}$
and 
$\lambda_{\vec{\theta'}}=\lambda_{\vec{\theta}}e^{-C}$,
$\overline{P}^3_{\vec{\theta'}}$ 
and $\lambda_{\vec{\theta'}}$ are
the Perron-Frobenius eigenvector and eigenvalue of the transition matrix ${W}_{\vec{\theta'}}$.
Then, we have 
$W_{\vec{\theta'}}=W_{\vec{\theta}}$.
\end{proof}

Now, we introduce the notation
${\cal W}_{\cX,W}:= \{V| V $ is a transition matrix and $
{\cal X}^2_{W}={\cal X}^2_V \}$.
Any element $W'\in {\cal W}_{\cX,W}$ can be written as
$W'(x|x')= W(x|x')e^{g(x,x')}$ by using an element 
$g \in {\cal G}({\cal X}^2)$
because of $\log \frac{W'(x|x')}{W(x|x')}\in {\cal G}({\cal X}^2)$.
Hence, if and only if 
the set of two-input functions $\{g_j\}$ form a basis of
the quotient space ${\cal G}({\cal X}^2)/{\cal N}_W({\cal X}^2)$,
the set ${\cal W}_{\cX,W}$ coincides with
the exponential family generated by $W$ with the generator $\{g_j\}$.
This fact shows that 
${\cal W}_{\cX,W}$ is an exponential family.

In particular, when $W$ is a positive transition matrix,
the subspace ${\cal N}_W({\cal X}^2)$ does not depend on $W$
and is abbreviated to ${\cal N}({\cal X}^2)$.
In this case,
${\cal W}_{\cX,W}$ is the set of positive transition matrices.
Then, it does not depend on $W$, and is abbreviated to ${\cal W}_{\cX}$.

We define the Fisher information matrix for the natural parameter by 
the Hesse matrix
$\mathsf{H}_{\vec{\theta}} [\phi]
:=[\frac{\partial^2 \phi}{\partial \theta^i\partial \theta^j}(\vec{\theta})]_{i,j}$.
The Fisher information matrix for the expectation parameter 
is given as 
$\mathsf{H}_{\vec{\theta}} [\phi]^{-1}$.
Further, 
for fixed values $\theta_o^{k+1},\ldots,\theta_o^d$,
we call the subset $\{W_{\vec{\theta}}\in {\cal E}|
\vec{\theta}=(\theta^1,\ldots,\theta^{k},\theta_o^{k+1},\ldots,\theta_o^d)
\}$ an exponential subfamily of ${\cal E}$.
The following are examples of an exponential family.

\begin{example}\Label{1-3-8}
Now, we assume that ${\cal X}=\{0,1, \ldots, m\}$ and 
$W$ is a positive transition matrix, i.e., ${\cal X}^2_{W}={\cal X}^2$.
Define $g_{i,j}(x,x')= \delta_{x,i}\delta_{x',j}$ for $i=1, \ldots,m$ and $j=0,1, \ldots,m$.
Then, the $m^2+m$ functions $g_{i,j}$ 
form a basis of the quotient space 
${\cal G}({\cal X}^2)/{\cal N}({\cal X}^2)$.
Therefore,
the set of positive transition matrices
forms an exponential family with the above choice of $g_{i,j}$. 
\end{example}
 
\begin{example}\Label{1-3-8b}
For a given subset ${\cal S} \subset {\cal X}^2$ for ${\cal X}=\{0,1, \ldots, m\}$, we choose a transition matrix $W$ whose support is ${\cal S} $.
Define the subset $\tilde{{\cal S} }$ as 
$ \{(i,j)\in S| i $ is not minimum integer satisfying 
$(i,j)\in {\cal S} $ for a fixed $j\}$.
We define $g_{i,j}(x,x')= \delta_{x,i}\delta_{x',j}$ for 
$(i,j)\in \tilde{{\cal S} }$.
Then,
the set ${\cal W}_{{\cal X},W}$ is an exponential family generated by
$\{g_{i,j}\}_{(i,j)\in \tilde{{\cal S} }}$.
However, the set ${\cal W}_{{\cal X},W}$ is not an exponential subfamily 
of the set of positive transition matrices
because it is not included in the set of positive transition matrices.
\end{example}

\begin{remark}
The above-defined exponential families contain
exponential families of distributions as follows.
For a given exponential family of distributions
$P_\theta$ on $\cX$ with the generator $f(x)$,
we define the transition matrix $W(x|x')$ as
$P_0(x)$ and the generator $g(x,x')$ as $f(x)$.
Then, the exponential family $W_{\theta}(x|x')$ is
$P_{\theta}(x)$.
The given potential function and the given expectation parameter (defined in the next subsection)
are the same as those in the case with the exponential family of distributions
$\{P_\theta\}$.
\end{remark}

\begin{remark}
The papers \cite{Feigin,KS,Hudson,Bhat} called 
a family of transition matrices $\{W_\theta(x|x')\}$ an exponential family
when $W_\theta(x|x')$ has the form
\begin{align}
W_\theta(x|x')=
e^{C(x,x')+\theta g(x,x')- \psi(\theta,x')}.\Label{1-3}
\end{align}
The papers \cite{Stefanov,Kuchler-Sorensen,Sorensen} extended the above definition to the continuous-time case.
However, our exponential family is written as \cite{HN}
\begin{align}
W_\theta(x|x')=
e^{C(x,x')+\theta g(x,x')+ \psi(\theta,x)- \psi(\theta,x')- \phi(\theta)}.
\Label{1-3b}
\end{align}
by choosing $C(x,x')$ and $\psi(\theta,x)$ 
as $\log W(x|x')$ and $\log \overline{P}^3_{\vec{\theta}}(x)$, respectively.
So, the traditional definition (\ref{1-3}) is different from ours.
The advantage of our model over their model is explained in Remark \ref{rem8-4}.

\end{remark}

\subsection{Mixture family}\Label{s4-5-2}
In the following, we assume that
the functions $\{g_j\}$ satisfies the condition of Lemma \ref{L1-14-2}.
For fixed values $\eta_{o,1},\ldots,\eta_{o,k}$,
we call the subset $\{W_{\vec{\theta}}\in {\cal E}|
\vec{\eta}(\vec{\theta})=(\eta_{o,1},\ldots,\eta_{o,k},\eta_{k+1},\ldots,\eta_d)
\}$ a {\it mixture subfamily of ${\cal E}$}.
Given a transition matrix $W$, 
real-valued functions $g_{j}$ on $\cX^2$, and real numbers $b_j$,
we say that the set $\{V \in {\cal W}_{\cX,W}| \sum_{x,x'}g_j(x,x')V(x|x')P_{V}(x')=b_j \forall j \}$ is a {\it mixture family on ${\cal X}^2_W$ 
generated by the constraints $\{g_j=b_j\}$}.
Note that a mixture family on ${\cal X}^2_W$ 
does not necessarily contain $W$
because its definition depends on the real numbers $b_j$.
When $W$ is a positive transition matrix, 
it is simply called a {\it mixture family generated by the constraints 
$\{g_j=b_j\}$}
because ${\cal W}_{\cX,W}$ is the set of positive transition matrices.
For a given transition matrix $W$ and 
two mixture families ${\cal M}_1$ and ${\cal M}_2$ on ${\cal X}^2_W$,
the intersection ${\cal M}_1 \cap {\cal M}_2$ is also a mixture family
on ${\cal X}^2_W$.

\begin{lemma}\Label{L5-1}
The intersection of the mixture family on ${\cal X}^2_W$ generated by the constraints $\{g_j=b_j\}_{j=1,\ldots,k}$
and the exponential family ${\cal W}_{\cX,W}$
is the mixture subfamily 
$\{W_{\vec{\theta}}\in {\cal W}_{\cX,W} |
\vec{\eta}(\vec{\theta})=(b_1,\ldots,b_{k},\eta_{k+1},\ldots,\eta_d)\}$
of the exponential family ${\cal W}_{\cX,W}$.
\end{lemma} 
Lemma \ref{L5-1} will be shown after Lemma \ref{lemma:markov-derivative-expectation} in Section \ref{s5}.
Here, we give examples for mixture families.

\begin{example}\Label{ex5}
A transition matrix $W$ on ${\cal X} \times {\cal Y}$ is called {\it non-hidden} for ${\cal X}$
when $W_X(x|x'):=\sum_{y \in {\cal Y}}W(x,y|x',y')$ does not depend on $y'\in {\cal Y}$.
For a transition matrix $W$ on $\cX \times \cY$,
the set ${\cal W}_{\cX|\cX\times \cY,W}
:=\{ V \in {\cal W}_{\cX\times \cY,W}|V$ is non-hidden for ${\cal X}$ on ${\cal X} \times {\cal Y}\}$ 
is a mixture family on $({\cal X}\times \cY)^2_W$.
Hence, the set 
${\cal W}_{\cX|\cX\times \cY,W} \cap {\cal W}_{\cY|\cX\times \cY,W}$
is also a mixture family on ${\cal X}^2_W$.
\end{example}

\begin{example}\Label{1-3-8c}
The set of bi-stochastic matrices on ${\cal X}=\{0,1, \ldots, m\}$ 
forms a mixture family as follows.
For a permutation $\sigma$, we define the transition matrix $W_{\sigma}(x|x')= \delta_{x,\sigma x'}$.
Then, we focus on the set $T$ of transpositions $(i,j)$
and the subset $H$ of cyclic permutations with length $3$
defined by $H:=\{(0,i,j)| 0<i<j\le m \}$.
Then, $|T\cup H|= |T| + | H|= \frac{m(m+1)}{2}+ \frac{m(m-1)}{2}=m^2$.
As will be shown in Appendix \ref{as2},
The set of bi-stochastic matrices on ${\cal X}=\{0,1, \ldots, m\}$ is parametrized as 
$\{ W_{\vec{\eta}} \}_{\vec{\eta} \in \Eta}$, where
\begin{align}
W_{\vec{\eta}}&:=
\sum_{\sigma \in T \cup H} \eta_{\sigma} W_\sigma 
+ (1- \sum_{\sigma \in T \cup H} \eta_{\sigma}) W_{id} \\
\Eta &:= \{\eta \in \mathbb{R}^{m^2}| 
W_{\vec{\eta}}(x|x') \ge 0 \hbox{ for } \forall x,x'\in {\cal X} \}.
\end{align}
We define the functions 
\begin{align}
g_i(x,x') &:= \delta_{x,i}-\delta_{x,0} \hbox{ for }i=1, \ldots,m \\
\hat{g}_\sigma(x,x') &:= W_\sigma (x|x')- W_{id}(x|x').
\end{align}
As will be shown in Appendix \ref{as2},
the set $\{g_i\}_{i=1}^{m}\cup \{\hat{g}_\sigma\}_{\sigma \in T \cup H}$
is linearly independent.
Then, the matrix $A=(a_{\sigma,\sigma'})$ given as follows is invertible:
\begin{align}
a_{\sigma,\sigma'}:=
\sum_{x,x'} \hat{g}_{\sigma'}(x,x') \hat{g}_\sigma (x,x')\frac{1}{m+1} .
\end{align}
Then, using the inverse matrix $B=A^{-1}$,
we can define the functions $\{g_\sigma\}_{\sigma \in T \cup H}$ 
as the dual basis in the following way:
\begin{align}
g_{\sigma'} :=
\sum_{\sigma \in T \cup H} b_{\sigma,\sigma'} \hat{g}_\sigma ,
\end{align}
which implies that
\begin{align}
\sum_{x,x'} g_{\sigma'}(x,x') \hat{g}_\sigma (x,x')\frac{1}{m+1} = \delta_{\sigma,\sigma'}.
\end{align}
Hence, the set of functions $\{g_i\}_{i=1}^m \cup \{g_\sigma\}_{\sigma \in T \cup H}$ is linearly independent.
We can employ the mixture parameter under the above set of functions.
Since the stationary distribution of $W_{\vec{\eta}}$ is the uniform distribution and
\begin{align}
\sum g_i(x,x') W_{\vec{\eta}}(x|x')\frac{1}{m+1} &= 0 \hbox{ for } i= 1, \ldots, m, \\
\sum g_\sigma(x,x') W_{\vec{\eta}}(x|x')\frac{1}{m+1} &= \eta_\sigma \hbox{ for } \sigma \in T \cup H, 
\end{align}
the transition matrix $W_{\vec{\eta}}(x|x')$ 
is the expectation parameter 
$(0, \ldots, 0, \eta_{\sigma})$.
That is, the set of bi-stochastic matrices on ${\cal X}$ 
is the mixture family generated by the constraints $\{g_j=0\}_{j=1,\ldots,m}$.
\end{example}

\subsection{Relation with relative entropy and relative R\'{e}nyi entropies}
The relative entropy and the relative R\'{e}nyi entropies are characterized by using the potential function $\phi(\vec{\theta})$
as follows. 
\begin{lemma}\Label{L7}
Two transition matrices 
${W}_{\vec{\theta}}$ and ${W}_{\vec{\theta}'}$ satisfies 
\begin{align}
D({W}_{\vec{\theta}} \| {W}_{\vec{\theta}'})= &
\sum_{j=1}^d
(\theta^j-{\theta'}^j)\frac{\partial \phi}{\partial \theta^j}(\vec{\theta})- \phi(\vec{\theta})+ \phi(\vec{\theta}') \Label{1-1}\\
D_{1+s}({W}_{\vec{\theta}} \| {W}_{\vec{\theta}'})=&
\frac{\phi((1+s)\vec{\theta} -s \vec{\theta}')-(1+s) \phi(\vec{\theta}) +s \phi(\vec{\theta}')}{s}.
 \Label{1-2}
\end{align}
\end{lemma}
\begin{proof}
Let $\varphi(1+s)$ be the logarithm of 
the Perron-Frobenius eigenvalue of the matrix
$W_{\vec{\theta}}(x|x')^{1+s}W_{\vec{\theta}'}(x|x')^{-s}$.
Then, we have $\varphi(1+s)=\phi((1+s)\vec{\theta} -s \vec{\theta}')-(1+s) \phi(\vec{\theta}) +s \phi(\vec{\theta}')$.  
Hence, we obtain (\ref{1-2}).
Taking the limit $s \to 0$, we obtain (\ref{1-1}).
\end{proof}

The Fisher information matrix 
$\mathsf{H}_{\vec{\theta}} [\phi]$
can be characterized by the limits of the 
relative entropy and relative R\'{e}nyi entropy as follows.
That is, taking the limits in \eqref{1-1} and \eqref{1-2} in Lemma \ref{L7}, we can show the following lemma.

\begin{lemma} \Label{L20}
For $\vec{c}=(c_1, \ldots, c_d)$, 
we have
\begin{align}
\Label{27-20}
\lim_{t \to 0}
\frac{2}{t^2}D({W}_{\vec{\theta}} \| {W}_{\vec{\theta}+\vec{c}t})
=&
\lim_{t \to 0}
\frac{2}{t^2}D({W}_{\vec{\theta}+\vec{c}t} \| {W}_{\vec{\theta}})
=\sum_{i,j}\mathsf{H}_{\vec{\theta}} [\phi]_{i,j}c^i c^j \\
\Label{27-21}
\lim_{t \to 0}
\frac{2}{t^2}D_{1+s}({W}_{\vec{\theta}} \| {W}_{\vec{\theta}+\vec{c}t})
=&
\lim_{t \to 0}
\frac{2}{t^2}D_{1+s}({W}_{\vec{\theta}+\vec{c}t} \| {W}_{\vec{\theta}})
=(1+s)\sum_{i,j}\mathsf{H}_{\vec{\theta}} [\phi]_{i,j}c^i c^j .
\end{align}
\end{lemma}

The right hand side of (\ref{1-1}) 
can be regarded as the Bregmann divergence \cite{Br}\footnote{Amari-Nagaoka \cite{AN} also defined the same quantity as 
the Bregmann divergence with the name ``canonical divergence.''
They showed that the canonical divergence satisfies 
the Pythagorean theorem and \eqref{5-28-1}
via the concept of the dually flat. 
Recently, Amari \cite{Am} showed these properties 
by a calculation of the convex function $\phi(\vec{\theta})$,
which does not require Christoffel symbols calculation.
Since the derivations by \cite{Am} 
more directly explain the relation 
between the convex function $\phi(\vec{\theta})$
and these properties, 
we refer the paper \cite{Am} for these properties.}
of the strictly convex function $\phi(\vec{\theta})$.
In the following, we derive several properties of the relative entropy
by using Bregmann divergence.
That is, the following properties follow only from the strong 
convexity of $\phi(\vec{\theta})$ and 
the properties of Bregmann divergence.

Using \cite[(40)]{Am}, we have
another expression of $ D({W}_{\vec{\theta}} \| {W}_{\vec{\theta}'})$
as 
\begin{align}
D(W_{\vec{\theta}(\vec{\eta})} \| W_{\vec{\theta}(\vec{\eta}')})
=\sum_j \theta(\vec{\eta}')^j(\eta_j'-\eta_j) 
- \nu(\vec{\eta}') +\nu(\vec{\eta}),
\Label{5-28-1}
\end{align}
where $\nu(\vec{\eta})$ is defined as Legendre transform of 
$\phi(\vec{\theta}) $ as
\begin{align*}
\nu(\vec{\eta})
&:= \max_{\vec{\theta}} \sum_i \theta^i \eta_i - \phi(\vec{\theta}) 
= \sum_i \theta^i(\vec{\eta}) \eta_i - \phi(\vec{\theta}(\vec{\eta})) .
\end{align*}
Since $\nu(\vec{\eta})$ is convex as well as $\phi(\vec{\theta})$,
we have the following lemma.
\begin{lemma}\Label{L1-14}
(1) For a fixed $\vec{\theta}$, 
the maps $\vec{\theta}' \mapsto
D(W_{\vec{\theta}} \| W_{\vec{\theta}'})$
and $D_{1+s}(W_{\vec{\theta}} \| W_{\vec{\theta}'})$
are convex
for $s > 0$.
(2) For a fixed $\vec{\theta}'$, 
the map $\vec{\eta} \mapsto
D(W_{\vec{\theta}(\vec{\eta})} \| W_{\vec{\theta}'})$ is convex.
\end{lemma}

\subsection{Pythagorean theorem}
It is known that Bregmann divergence satisfies the Pythagorean theorem for \cite[(34)]{Am}.
Applying this fact, we have the following proposition as the Pythagorean theorem.
\begin{proposition}(Nagaoka \cite[(23)]{HN})\Label{P1-15}
We focus on two points 
$\vec{\theta}'=({\theta'}^1,\ldots,{\theta'}^d)$
and $\vec{\theta}''=({\theta''}^1,\ldots,{\theta''}^d)$.
We choose 
the exponential subfamily of ${\cal E}$ whose natural parameters
$\theta^{k+1},\ldots,\theta^d $ are fixed to 
${\theta''}^{k+1},\ldots,{\theta''}^d $,
and
the mixture subfamily of ${\cal E}$ whose expectation parameters
$\eta_1,\ldots,\eta_k $ are fixed to 
$\eta(\vec{\theta}')^1,\ldots,\eta(\vec{\theta}')^k $.
Let $\tilde{\vec{\theta}}=(\tilde{\theta}^1,\ldots,\tilde{\theta}^d)$
be the natural parameter of the intersection of these two subfamilies of ${\cal E}$.
That is, 
$\tilde{\theta}^{j}={\theta''}^{j}$ for $j=k+1,\ldots,d$
and
$\eta_j(\tilde{\vec{\theta}})=\eta_j(\vec{\theta}') $ for $k=1, \ldots,k$.
Then, we have 
\begin{align}
\Label{5-1}
D({W}_{\vec{\theta}'} \| {W}_{\vec{\theta}''})
=
D({W}_{\vec{\theta}'} \| {W}_{\tilde{\vec{\theta}}})+D({W}_{\tilde{\vec{\theta}}} \| {W}_{\vec{\theta}''}).
\end{align}
\end{proposition}
Indeed, Nagaoka \cite{HN} showed (\ref{5-1}) 
in a more general form
by showing the dually flat structure \cite{AN} 
via Christoffel symbols calculation.
Using (\ref{5-1}) and Lemma \ref{L5-1}, we obtain 
the following corollary.
\begin{corollary}\Label{T5-1}
Given a transition matrix $V$ and 
a mixture family ${\cal M}$ on ${\cal X}^2_V$ with 
constraints $\{g_j=b_j\}_{j=1}^k$,
we define 
$V^*:= \argmin_{W \in {\cal M}}D(W\|V)$.

(1) Any transition matrix $W \in {\cal M}$ satisfies
$D(W\|V)=D(W\|V^*)+D(V^*\|V)$.

(2) The transition matrix $V^*$ is the intersection of 
the mixture family ${\cal M}$ on ${\cal X}^2_V$ 
and the exponential family generated by $V$ and the generator  
$\{g_j\}_{j=1}^k$.
\end{corollary}

\begin{proof}
First, we notice that
the exponential family ${\cal W}_{\cX,V}$
contains $V$ and includes ${\cal M}$.
Choose an element $\tilde{V}$ in the intersection of 
the mixture family ${\cal M}$ on ${\cal X}^2_V$ 
and the exponential family ${\cal E}_{V}$ generated by $V$ and 
the generator $\{g_j\}_{j=1}^k$.
We apply (\ref{5-1}) to 
the mixture family ${\cal M}$ and
the exponential family ${\cal E}_{V}$.
Then, any transition matrix $W \in {\cal M}$ satisfies that
$D(W\|V)=D(W\|\tilde{V})+D(\tilde{V} \|V)$.
Since $D(W\|\tilde{V})> 0$ except for $W=\tilde{V}$, 
we have $ \min_{W \in {\cal M}}D(W\|V)= D(\tilde{V}\|V)$,
which implies that $V^*=\tilde{V}$, i.e., (2).
Hence, we obtain (1).
\end{proof}

Similarly, we have another version of the above corollary.
\begin{corollary}\Label{C15-1}
Given a transition matrix $W$
and an exponential family ${\cal E} \subset {\cal W}_{\cX,W}$ 
with the generator $\{g_j\}$,
we define 
$W_*:= \argmin_{V \in {\cal E}}D(W\|V)$.
Assume that $\sum_{x,x'} g_j(x,x')W_*(x|x')P_{W_*}(x')=b_j$.

(1) Any transition matrix $V \in {\cal E}$ satisfies
$D(W\|V)=D(W\|W_*)+D(W_*\|V)$.

(2) The transition matrix $W_*$ is the intersection of 
the exponential family ${\cal E}$ and 
the mixture family on ${\cal X}^2_W$ with the constraints $\{g_j=b_j\}$.
\end{corollary}

\begin{example}\Label{1-3-8d}
We choose transition matrices $V_X$ and $V_Y$ on 
${\cal X}$ and ${\cal Y}$, respectively.
We also choose a transition matrix $W$ on ${\cal X} \times {\cal Y}$ whose 
support is $({\cal X} \times {\cal Y})^2_{V_X \times V_Y}$.
When a set of two-input functions $\{g_{X|i}\}$ forms a basis of
${\cal G}(\cX^2)/{\cal N}_{V_X}(\cX^2)$,
the exponential family 
generated by $V_X\times V_Y$ with the generator $\{g_{X|i}\}$ is 
$\{ V_X'\times V_Y| V_X' \in {\cal W}_{{\cal X},V_X} \}$.
When a set of two-input functions $\{g_{Y|j}\}$ forms a basis of
${\cal G}(\cY^2)/{\cal N}_{V_Y}(\cY^2)$,
the exponential family 
generated by $V_X\times V_Y$ with the generator $\{g_{X|i}\} \cup \{g_{Y|j}\}$
is $\{ V_X' \times V_Y'| V_Y' \in {\cal E}_{{\cal X},V_X} ,V_Y'\in {\cal E}_{{\cal Y},V_Y} \}$.
Hence, 
when a transition matrix $W$ belongs to 
a mixture family with the constraints 
$\{g_{X|i}=a_i\} \cup \{g_{Y|j}=b_j\}$,
the intersection between the exponential family 
and the mixture family 
consists of one points, which is denoted by $W_X'\times W_Y'$.
Applying (\ref{5-1}), we obtain 
\begin{align}
D(W\| V_X \times V_Y)=
D(W\| W_X' \times W_Y')+ D(W_X' \times W_Y' \| V_X \times V_Y).
\end{align}
In particular, when $W$ is non-hidden for ${\cal X}$ (for the definition, see Example \ref{ex5}.),
$W_X$ satisfies the same constraint $\{g_{X|i}=a_i\}$
because the stationary distribution $P_{W_X}$ is the marginal distribution of the stationary distribution $P_{W}$.
Hence, $W_X'=W_X$.
Thus, $W_X'$ can be regarded as a marginalization of 
a transition matrix $W$ that is not necessarily non-hidden.
\end{example}



\section{Stationary two-observation case}\Label{s5}
\subsection{Relative entropies and expectation}
In the previous section, we formally defined several information quantities
from the convex function $\phi(\vec{\theta})$ in the multi-parameter case.
In this section, we consider the relation with the structure of probabilities
in the one-parameter case.
That is, we will see how the information quantities reflect 
the conventional information quantities.  
For this purpose, 
we assume that the input distribution is the stationary distribution of the given 
transition matrix.

Since the stationary distribution of the given transition matrix ${W}_{\theta}$
is $\overline{P}^1_{\theta}$ given in \eqref{5-6},
we can define the joint distribution 
\begin{align}
 W_{\theta} \times \overline{P}^1_{\theta}(x,x'):= 
W_{\theta}(x|x')\overline{P}^1_{\theta}(x') 
= 
\frac{\overline{P}^3_{\theta}(x) \overline{W}_{\theta}(x|x') \overline{P}^2_{\theta}(x')}
{\lambda_{\theta}
\sum_{x''}\overline{P}^3_{\vec{\theta}}(x'')\overline{P}^2_{\vec{\theta}}(x'')}
\end{align}
on $\cX \times \cX$.
Now, we focus on the probability distribution family 
$\{ W_{\theta} \times \overline{P}^1_{\theta}\}$,
and denote the expectation and the variance under the distribution $W_{\theta} \times \overline{P}^1_{\theta}$
by $\mathsf{E}_\theta$ and $\mathsf{V}_\theta$.
These are simplified to $\mathsf{E}$ and $\mathsf{V}$ when $\theta=0$.
\begin{lemma}(\cite[Theorem 4]{HN}, \cite[(28)]{NK}) \Label{lemma:markov-derivative-expectation}
For $\theta \in \mathbb{R}$, we have
\begin{align}
\eta(\theta)=\frac{d \phi}{d\theta}(\theta) = \mathsf{E}_\theta[g(X,X')] 
= \sum_{x,x^\prime} \overline{P}^1_\theta(x) W_\theta(x|x^\prime) g(x,x^\prime).
\end{align}
\end{lemma}

The lemma shows the reason why we call the parameter $\eta$ the expectation parameter.

\begin{proof}
From the definition of $W_\theta$, we have
\begin{eqnarray} \label{eq:taking-derivative-1}
\frac{d}{d\theta} \log W_\theta(x|x^\prime) = - \frac{d}{d\theta} \log \lambda_\theta + \frac{d}{d\theta} \log \frac{\overline{P}^3_\theta(x)}{\overline{P}^3_\theta(x^\prime)} + g(x,x^\prime). 
\end{eqnarray}
Taking the average of the both hand sides
with respect to the distribution $W_{\theta} \times \overline{P}^1_{\theta}$, we have
$0 = - \frac{d}{d\theta} \log \lambda_\theta + \sum_{x,x^\prime} \overline{P}^1_\theta(x) W_\theta(x|x^\prime) g(x,x^\prime).$
\end{proof}

Lemma \ref{lemma:markov-derivative-expectation} shows Lemma \ref{L5-1} as follows.

\begin{proofof}{Lemma \ref{L5-1}}
In this proof, we consider the multi-parameter case.
Replacing the derivative by the partial derivative in Lemma \ref{lemma:markov-derivative-expectation},
we have 
\begin{align}
\eta_j(\vec{\theta})=\frac{\partial \phi}{\partial \theta_j}(\vec{\theta})
=\sum_{x,x^\prime} \overline{P}^1_{\vec{\theta}}(x) W_{\vec{\theta}}(x|x^\prime) g_j(x,x^\prime).
\Label{1-3-9}
\end{align}
Choose the generator $\{g_1, \ldots, g_k\}$ 
of the mixture family on ${\cal X}^2_W$.
There exist two-input functions $g_{k+1},...,g_l$ such that
the set of two-input functions $\{g_1, \ldots, g_l\}$ form
a basis of ${\cal G}({\cal X}^2)/{\cal N}_W({\cal X}^2)$.
Hence, due to (\ref{1-3-9}), we see that 
the intersection of the mixture family on ${\cal X}^2_W$ generated by the constraints $\{g_j=b_j\}_{j=1,\ldots,k}$
and the exponential family ${\cal W}_{\cX,W}$
is the mixture subfamily 
$\{W_{\vec{\theta}}\in {\cal W}_{\cX,W} |
\vec{\eta}(\vec{\theta})=(b_1,\ldots,b_{k},\eta_{k+1},\ldots,\eta_d)\}$
of the exponential family ${\cal W}_{\cX,W}$.
\end{proofof}

Now, we introduce the conditional relative entropy 
for transition matrices $W$ and $V$ from $\cX$ to $\cY$
and a distribution $P$ on $\cX$
as follows.
\begin{align*}
D(W \| V |P) &:= D(W\times P \| V\times P ) ,
\end{align*}
where the relative entropy between two distributions $P$ and $P'$
is defined in the conventional way as
$D(P\|P') :=
\sum_x P(x) \log \frac{P(x)}{P'(x)} .
$
Hence, the relative entropy 
defined in the previous section
is characterized as follows \cite[(24)]{HN}.
\begin{align}
\nonumber
& D(W_\theta \| W_{\theta'} )
= (\theta -\theta') \frac{d \phi}{d\theta}(\theta) - \phi(\theta) +\phi(\theta')  \\
=& \sum_{x,x^\prime} \overline{P}^1_\theta(x^\prime) 
W_\theta(x|x^\prime) 
\log \frac{\overline{W}_{\theta}(x|x')}{\overline{W}_{\theta'}(x|x')}
- \phi(\theta) +\phi(\theta')  \nonumber \\
=& \sum_{x,x^\prime} \overline{P}^1_\theta(x^\prime) 
W_\theta(x|x^\prime) 
\log \frac{{W}_{\theta}(x|x')}{{W}_{\theta'}(x|x')}
-\log \frac{\overline{P}^3_{\theta}(x)}{\overline{P}^3_{\theta'}(x)}
+\log \frac{\overline{P}^3_{\theta}(x')}{\overline{P}^3_{\theta'}(x')}
\nonumber \\
\stackrel{(a)}{=}& \sum_{x,x^\prime} \overline{P}^1_\theta(x^\prime) 
W_\theta(x|x^\prime) \log \frac{W_\theta(x|x^\prime)}{W_{\theta'}(x|x^\prime)} 
=D(W_\theta \| W_{\theta'} |\overline{P}^1_\theta) ,
\Label{eq:markov-divergence-cdf-relation}
\end{align}
where $(a)$ follows from the fact that
$\overline{P}^1_\theta(x)=\sum_{x'}W_\theta(x|x^\prime)\overline{P}^1_\theta(x^\prime) $.

\begin{proofof}{\eqref{th15-2}}
Since the map $W' \mapsto 
-\sum_{x,x^\prime} \overline{P}^1_\theta(x^\prime) 
W_\theta(x|x^\prime) \log W'(x|x^\prime)
$ is convex for a given $\theta$,
\eqref{eq:markov-divergence-cdf-relation} guarantees 
\eqref{th15-2}.
\end{proofof}


\subsection{Fisher information and variance}
Using the Fisher information $J_{\theta}^1$ of the family $\{ \overline{P}^1_{\theta}\}_{\theta}$
of stationary distributions,
we discuss the Fisher information 
$J_{\theta}^2$
of the family $ \{ W_{\theta} \times \overline{P}^1_{\theta}\}_{\theta}$ of joint distributions
in the following lemma.
\begin{lemma} \Label{L11}
The Fisher information $J_{\theta}^2$ can be written as
\begin{align}
J_{\theta}^2 = \frac{d^2\phi}{d\theta^2} (\theta) + J_{\theta}^1. \Label{27-10}
\end{align}
\end{lemma}

\begin{lemma} \Label{L11-2}
The second derivative $\frac{d^2\phi}{d\theta^2} (\theta)$ is calculated as
\begin{align}
\frac{d^2\phi}{d\theta^2} (\theta)
=
\mathsf{V}_\theta \Bigl[
g(X,X')- \frac{d \phi}{d \theta}(\theta)
+ \frac{d}{d\theta} \log \overline{P}^3_{\theta}(X)
- \frac{d}{d\theta} \log \overline{P}^3_{\theta}(X') \Bigr].\Label{27-11}
\end{align}
In particular, when $\theta=0$,
\begin{align}
\frac{d^2\phi}{d\theta^2} (0)
=
\mathsf{V}_0 [g(X,X')]
+
2 \sum_{x,x'}
W(x|x') g(x,x') \frac{d \overline{P}^2_{\theta}(x')}{d \theta}\Bigr|_{\theta=0}.
\Label{27-12}
\end{align}
\end{lemma}

Proofs of Lemmas \ref{L11} and \ref{L11-2}
are given in Appendix \ref{as3}.
Further, the quantity $\frac{d^2 \phi}{d \theta^2}(0)$ has another form \cite[Theorem 6.6]{HW14-2}.
Using Lemma \ref{L11-2}, we can show Lemma \ref{L1} as follows.

\begin{proofof}{Lemma \ref{L1}}
Due to (\ref{27-11}), the non-negativity of variance implies that 
$\phi(\theta)$ is convex.
Since Condition (2) trivially implies Condition (3),
it is enough to show that
Condition (1) implies Condition (2) and Condition (3) implies Condition (1).

Assume Condition (1).
Then, the random variable $g(X,X')- \frac{d \phi}{d \theta}(\theta)
+ \frac{d}{d\theta} \log \overline{P}^3_{\theta}(X)
- \frac{d}{d\theta} \log \overline{P}^3_{\theta}(X')
$ is not a constant on $\cX^2_W$.
Hence, the variance in (\ref{27-11}) is strictly greater than zero,
which implies Condition (2).

Conversely, we assume that Condition (1) does not hold, i.e.,
$g(x,x')=f(x)-f(x')+C$ for any $(x,x')\in \cX^2_W$ with a constant 
$C \in \mathbb{R}$.
Then, we can find that the Perron-Frobenius eigenvalue of 
$\overline{W}_\theta(x|x^\prime) = W(x|x^\prime) e^{\theta f(x) - \theta f(x^\prime) + \theta C}$ 
is $\lambda_\theta = e^{\theta C}$ and
its right eigenvector is $\overline{P}^2_\theta$.  
Thus, we have $\frac{d^2\phi(\theta)}{d\theta^2} =0$, i.e.,
Condition (3) does not hold.
Hence, Condition (3) implies Condition (1).
\end{proofof}


\section{Stationary $n+1$-observation case}\Label{s6}
\subsection{Information quantities}
Similar to the previous section, this section also discusses the one-parameter case with the stationary initial distribution $\overline{P}^1_{\theta}$.
Now, we consider the distribution 
$W^{\times n}_\theta \times \overline{P}^1_\theta$ on $\cX^n$, which is defined as
\begin{align}
W^{\times n}_\theta \times \overline{P}^1_\theta(x_n,\ldots, x_1)
:= W_\theta(x_{n+1}|x_{n})\cdots W_\theta(x_2|x_{1}) \overline{P}^1_\theta(x_1).
\end{align}
We also define the random variable $g^n(X^{n+1})
:=\sum_{k=1}^{n}g(X_{k+1},X_k)$ for $X^{n+1}:=(X_{n+1},\ldots, X_1)$.
In this section,
we denote the expectation and the variance under the distribution $W_{\theta}^{n} \times \overline{P}^1_{\theta}$
by $\mathsf{E}_\theta$ and $\mathsf{V}_\theta$.
Then, the cumulant generating function $\phi_{n}(\theta):=\log \mathsf{E}_0 [\exp (\theta g^n(X^{n+1}))]$ satisfies 
\begin{align}
\frac{ d\phi_{n}}{d \theta}(\theta) =& 
\mathsf{E}_\theta [g^n(X^{n+1})] = n \eta(\theta) .
\Label{25-7}
\end{align} 

Now, we calculate information quantities.
Similar to Lemma \ref{L11}, the Fisher information can be calculated as follows.
\begin{lemma}\Label{L11-11}
The Fisher information $J^{n+1}_{\theta}$ of the family $\{W^{\times n}_\theta \times \overline{P}^1_\theta\}_{\theta}$ 
can be written as
\begin{align}
J^{n+1}_{\theta}= n \frac{d^2 \phi}{d\theta^2}(\theta)+ J^1_{\theta}.
\end{align} 
\end{lemma}
The proof can be done in the same way as Lemma \ref{L11}.
The conditional relative entropy is characterized by the Bregman divergence defined by the convex function $\phi(\theta)$ as follows.
\begin{align}
\nonumber 
&D(W_\theta^{\times n} \| W_{\theta'}^{\times n} |\overline{P}^1_\theta) 
:=
D(W_\theta^{\times n} \times \overline{P}^1_\theta\| W_{\theta'}^{\times n} \times \overline{P}^1_\theta) \\
=& n
((\theta -\theta') \frac{d \phi}{d\theta}(\theta) - \phi(\theta) + \phi(\theta'))
=n D(W_\theta \| W_{\theta'} ).
\label{25-6} 
\end{align}

\subsection{Asymptotically efficient estimator}
The relation \eqref{25-7} implies that $\frac{g^n(X^{n+1})}{n}$ is an unbiased estimator for the parameter $\eta$. 
The variance of $g^n(X^{n+1})$ is evaluated as follows.
\begin{lemma}\Label{L20b}
The inequalities
\begin{align}
n \frac{d^2 \phi}{d\theta^2} (\theta)
(1-2\sqrt{\frac{\hat{\mathsf{V}}_\theta}{ n \frac{d^2 \phi}{d\theta^2} (\theta)}})^2
\le \mathsf{V}_\theta [g^n(X^{n+1})]
 \le
n \frac{d^2 \phi}{d\theta^2} (\theta)
(1+2\sqrt{\frac{\hat{\mathsf{V}}_\theta}{ n \frac{d^2 \phi}{d\theta^2} (\theta)}})^2
\Label{25-5}
\end{align}
hold, where
$\hat{\mathsf{V}}_\theta:=
\mathsf{V}_\theta[\frac{d}{d\theta} \log \overline{P}^3_{\theta}(X)]
=
\sum_{x} \overline{P}^1_{\theta}(x)
(\frac{d}{d\theta} \log \overline{P}^3_{\theta}(x))^2$.
\end{lemma}

Hence, we obtain
\begin{align}
\mathsf{V}_\theta [\frac{g^n(X^{n+1})}{n}] 
=
\frac{\mathsf{V}_\theta [g^n(X^{n+1})] }{n^2}
= \frac{\frac{d^2\phi}{d\theta^2} (\theta)}{n} +O(\frac{1}{n\sqrt{n}}).
\Label{25-8}
\end{align}

The Fisher information $\tilde{J}^{n+1}_{\eta(\theta)}$
for the expectation parameter $\eta$ 
of the family $\{ W_\theta^{\times n} \times \overline{P}^1_{\theta}\}_{\theta}$ is 
\begin{align*}
\tilde{J}^{n+1}_{\eta(\theta)}=&
\frac{{J}^{n+1}_{\theta}} 
{(\frac{d \eta(\theta)}{d \theta})^{2}}
=
(n \frac{d^2\phi}{d\theta^2} (\theta)
+ J_{\theta}^1)
(\frac{d^2\phi}{d\theta^2} (\theta))^{-2} 
=
\frac{n(1+ \frac{J_{\theta}^1}{n\frac{d^2\phi}{d\theta^2} (\theta)})}
{(\frac{d^2\phi}{d\theta^2} (\theta))}. 
\end{align*}
That is,
the lower bound of the variance of the unbiased estimator 
given by Cram\'{e}r-Rao inequality is
$\frac{d^2\phi}{d\theta^2} (\theta)/
n(1+ \frac{J_{\theta}^1}{n\frac{d^2\phi}{d\theta^2} (\theta)})$.
Hence, any unbiased estimator 
$Z_n$
for the expectation parameter $\eta$ 
satisfies
\begin{align}
\mathsf{V}_\theta [Z_n] 
\ge
\frac{\frac{d^2\phi}{d\theta^2} (\theta)}{
n(1+ \frac{J_{\theta}^1}{n\frac{d^2\phi}{d\theta^2} (\theta)})}
=\frac{\frac{d^2\phi}{d\theta^2} (\theta)}{n}
-\frac{J_{\theta}^1}{n^2}+ o(\frac{1}{n^2})
\Label{25-6f}.
\end{align}
The relation (\ref{25-8}) shows that 
the unbiased estimator $\frac{g^n(X^{n+1})}{n}$ realizes the optimal performance with the order $\frac{1}{n}$.

\begin{proofof}{Lemma \ref{L20b}}
The combination of (\ref{25-1}) and (\ref{25-2}) implies that
$\frac{d^2 \phi}{d\theta^2} (\theta)$
is the variance of 
$[- \frac{d \phi}{d \theta}(\theta)
+ \frac{d}{d\theta} \log \overline{P}^3_{\theta}(X)
- \frac{d}{d\theta} \log \overline{P}^3_{\theta}(X')
+ g(X,X')]$
under the distribution $W_\theta\times \overline{P}^1_\theta$
in the two-observation case.
In the $n+1$-observation case, using Lemma \ref{L11-11}, we can similarly show that
$n \frac{d^2 \phi}{d\theta^2} (\theta)$
is the variance of 
$[- n \frac{d \phi}{d \theta}(\theta)
+ \frac{d}{d\theta} \log \overline{P}^3_{\theta}(X_{n+1})
- \frac{d}{d\theta} \log \overline{P}^3_{\theta}(X_1)
+ g^n(X^{n+1})]$
under the distribution $W_\theta^n \times \overline{P}^1_\theta$.

Now, we define the $2$-norm of the random variable $f(X^{n+1})$ as
$\|f\|_2:= 
\sqrt{\sum_{x^{n+1}} W_\theta^n \times P_\theta(x^{n+1}) f(x^{n+1})^2}$.
Then, we have
\begin{align*}
& \sqrt{n \frac{d^2 \phi}{d\theta^2} (\theta)}=
\|g^n(X^{n+1})- n \frac{d \phi}{d \theta}(\theta)
+ \frac{d}{d\theta} \log \overline{P}^3_{\theta}(X_{n+1})
- \frac{d}{d\theta} \log \overline{P}^3_{\theta}(X_1)
\|_2 \nonumber \\
\le &
\| g^n(X^{n+1})- n \frac{d \phi}{d \theta}(\theta)\|_2
+\|\frac{d}{d\theta} \log \overline{P}^3_{\theta}(X_{n+1})\|_2
+\|\frac{d}{d\theta} \log \overline{P}^3_{\theta}(X_{1})\|_2 \\
= &
\sqrt{\mathsf{V}_\theta [g(X^{n+1})]}
+
2\sqrt{\hat{\mathsf{V}}_\theta},
\end{align*}
which implies 
$ (\sqrt{n \frac{d^2 \phi}{d\theta^2} (\theta)}-2\sqrt{\hat{\mathsf{V}}_\theta})^2
\le \mathsf{V}_\theta [g(X^{n+1})]$.
Then, we obtain the first inequality
because 
$n \frac{d^2 \phi}{d\theta^2} (\theta)
(1-2\sqrt{\frac{\hat{\mathsf{V}}_\theta}{ n \frac{d^2 \phi}{d\theta^2} (\theta)}})^2
=
(\sqrt{n \frac{d^2 \phi}{d\theta^2} (\theta)}-2\sqrt{\hat{\mathsf{V}}_\theta})^2$.
Similarly,
since 
$\|g^n(X^{n+1})- n \frac{d \phi}{d \theta}(\theta)
+ \frac{d}{d\theta} \log \overline{P}^3_{\theta}(X_{n+1})
- \frac{d}{d\theta} \log \overline{P}^3_{\theta}(X_1)
\|_2
\ge
\|g^n(X^{n+1})- n \frac{d \phi}{d \theta}(\theta)\|_2
-\|\frac{d}{d\theta} \log \overline{P}^3_{\theta}(X_{n+1})\|_2
-\|\frac{d}{d\theta} \log \overline{P}^3_{\theta}(X_1)\|_2$,
we obtain the second inequality
because 
$\frac{d^2 \phi}{d\theta^2} (\theta)
(1+2\sqrt{\frac{\hat{\mathsf{V}}_\theta}{ n \frac{d^2 \phi}{d\theta^2} (\theta)}})^2
=
(\sqrt{n \frac{d^2 \phi}{d\theta^2} (\theta)}+2\hat{\mathsf{V}}_\theta)^2$.
\end{proofof}

\section{Non-stationary $n+1$-observation case}\Label{s7}
Similar to the previous section, this section also discusses the one-parameter case.
Now, we consider the non-stationary case.
Since the convergence to the stationary distribution is required,
we assume that the transition matrices $W_{\theta}$ are ergodic as well as irreducible.
Then, we fix an arbitrary initial distributions $P_\theta$ on $\cX$ such that
the $P_\theta$ is distribution is smoothly parameterized by the parameter $\theta$.
In this section, 
we assume that $W_\theta$ is the exponential family generated by the generator $g(x,x')$
and the random variable $X^{n+1}:= (X_{n+1}, \ldots,X_1)$
is subject to 
$W_{\theta}^{\times n} \times P_\theta$ with the unknown parameter $\theta$.
Then, we denote the expectation and the variance under the distribution $W_{\theta}^{\times n} \times P_\theta$
by $\mathsf{E}_\theta$ and $\mathsf{V}_\theta$.
In this  general case, the relation (\ref{25-6}) does not hold.
In stead of these relations, as is shown in \cite[Lemma 5.4]{HW14-2}, 
we have
\begin{align}
\Label{1-4-1}
\lim_{n \to \infty}
\frac{1}{n}D(W_{\theta}^{\times n} \times P_{\theta} 
\|W_{\theta'}^{\times n} \times P_{\theta'})
=&
D(W_{\theta}\|W_{\theta'}), \\
\Label{1-4-2}
\lim_{n \to \infty}
\frac{1}{n}D_{1+s}(W_{\theta}^{\times n} \times P_{\theta} 
\|W_{\theta'}^{\times n} \times P_{\theta'})
=&
D_{1+s}(W_{\theta}\|W_{\theta'}).
\end{align}

For a function $h$ on $\mathbb{R}$,
we define the random variable $\tilde{g}^n (X^{n+1})
:= 
{g}^n(X_{n+1})+h(X_1)$.
When we use the random variable $\tilde{g}^n (X^{n+1})/n$ as an estimator of the parameter $\eta(\theta)$, the error is measured by
the mean square error:
\begin{align}
\mathsf{MSE}_\theta [\tilde{g}^n (X^{n+1})]:=
\mathsf{E}_\theta [(\frac{\tilde{g}^n (X^{n+1})}{n}-\eta(\theta))^2 ].
\end{align}
Then,
we have
$\mathsf{E}_\theta[\tilde{g}^n(X^{n+1})]= \mathsf{E}_\theta [g^n(X^{n+1})]+
\mathsf{E}_\theta [h(X_1)]$.
In the following discussion, we employ 
the norm $\|f(X^{n+1})\|_2 :=\sqrt{\mathsf{E}_\theta [f(X^{n+1})^2]}$
for a function $f$ on $\mathbb{R}^{n+1}$.
Using the triangle inequality for this norm, 
we have
\begin{align}
&\sqrt{\mathsf{V}_\theta [g^n(X^{n+1})]}-\sqrt{\mathsf{V}_\theta [h(X_1)]}
\le 
\sqrt{\mathsf{V}_\theta[\tilde{g}^n(X^{n+1})]} \nonumber \\
\le &
\sqrt{\mathsf{V}_\theta [g^n(X^{n+1})]}+\sqrt{\mathsf{V}_\theta [h(X_1)]},
\Label{11-11-2}
\\
& \sqrt{\mathsf{E}_\theta [( \frac{g^n (X^{n+1})}{n}
- \mathsf{E}_\theta [\frac{g^n (X^{n+1})}{n}])^2 ]}
- \sqrt{\mathsf{E}_\theta [(\frac{h (X_1)}{n}+ 
\mathsf{E}_\theta [ \frac{g^n (X^{n+1})}{n}] -\eta(\theta) )^2]} \nonumber \\
\le &
\sqrt{\mathsf{E}_\theta [( \frac{\tilde{g}^n (X^{n+1})}{n}-\eta(\theta))^2 ]}
\nonumber \\
\le &
\sqrt{\mathsf{E}_\theta [( \frac{g^n (X^{n+1})}{n}
- \mathsf{E}_\theta [\frac{g^n (X^{n+1})}{n}])^2 ]}
+ \sqrt{\mathsf{E}_\theta [(\frac{h (X_1)}{n}+ 
\mathsf{E}_\theta [ \frac{g^n (X^{n+1})}{n}] -\eta(\theta) )^2]}.
\Label{11-11-3}
\end{align}
It is known that 
the expectation of $g^n(X^{n+1})$
and the variance of $\frac{g^n(X^{n+1})}{\sqrt{n}}$
converge to 
those under the stationary distribution \cite{kemeny-snell-book,DZ}.
Hence, due to (\ref{25-7}) and (\ref{25-8}), we have
\begin{align}
&\lim_{n\to \infty} \mathsf{E}_\theta [\frac{\tilde{g}^n(X^{n+1})}{n}] 
=
\lim_{n\to \infty} \mathsf{E}_\theta [\frac{g^n(X^{n+1})}{n}] 
=
\eta(\theta) =\frac{ d\phi}{d \theta}(\theta) ,\Label{27-15}\\
& \lim_{n\to \infty} n \mathsf{MSE}_\theta [\frac{\tilde{g}^n (X^{n+1})}{n}]
\stackrel{(a)}{=} 
 \lim_{n\to \infty} \mathsf{V}_\theta [\frac{\tilde{g}^n(X^{n+1})}{\sqrt{n}}] 
\nonumber
 \\
\stackrel{(b)}{=} 
 &
\lim_{n\to \infty} \mathsf{V}_\theta [\frac{g^n(X^{n+1})}{\sqrt{n}}] 
= 
\frac{d^2\phi}{d\theta^2} (\theta),
\Label{27-15b}
\end{align}
where $(a)$ and $(b)$ follow from \eqref{11-11-3} and \eqref{11-11-2}, respectively.
The relation (\ref{27-15}) shows that the estimator $\frac{\tilde{g}^n(X^{n+1})}{n}$
is asymptotically unbiased for the parameter $\eta$.
The mean square error is
$\frac{d^2 \phi}{d\theta^2}(\theta) \frac{1}{n}+
o(\frac{1}{n})$, which implies (\ref{1-15-1}).
Further, it is shown that the random variable 
$\sqrt{n}(\frac{{g}^n(X^{n+1})}{n}-\eta(0))$ asymptotically 
obeys the Gaussian distribution with the variance 
$\frac{d^2 \phi}{d\theta^2}(0)$ at $\theta=0$ \cite[Corollary 6.2]{HW14-2}.
Replacing $W_{0}$ by $W_{\theta}$, we find that 
the random variable 
$\sqrt{n}(\frac{{g}^n(X^{n+1})}{n}-\eta(\theta))$ asymptotically 
obeys the Gaussian distribution with the variance 
$\frac{d^2 \phi}{d\theta^2}(\theta)$.

Next, for the family $\{W^{\times n}_\theta \times P_\theta\}_{\theta}$, 
we consider the Fisher information $J_\theta^n$ for the natural parameter $\theta$
and the Fisher information $\tilde{J}_\theta^n$ for the expectation parameter $\eta$.
\begin{lemma}\Label{L29-1}
The limit of the Fisher information $J_\theta^n$ 
for the natural parameter $\theta$
is characterized as
\begin{align}
\lim_{n\to \infty}\frac{J_\theta^n}{n}=
\frac{d^2 \phi}{d\theta^2}(\theta).
\Label{29-1}
\end{align}
Hence, the limit of the Fisher information $\tilde{J}_\theta^n$ 
for the expectation parameter $\eta$ is characterized as
$\lim_{n\to \infty}\frac{\tilde{J}_\theta^n}{n}=
\frac{d^2 \phi}{d\theta^2}(\theta)^{-1}$.
\end{lemma}
Lemma \ref{L29-1} implies that
the lower bound of the Cram\'{e}r-Rao inequality is 
$\frac{d^2 \phi}{d\theta^2}(\theta(\eta)) \frac{1}{n}+
o(\frac{1}{n})$.
Therefore, 
the estimator $\frac{\tilde{g}^n(X^{n+1})}{n}$ attains the lower bound by the Cram\'{e}r-Rao inequality
with the order $\frac{1}{n}$.
That is, the estimator $\frac{\tilde{g}^n(X^{n+1})}{n}$ is asymptotically efficient.

\begin{proofof}{Lemma \ref{L29-1}}
Similar to (\ref{25-2}), we have
\begin{align}
\nonumber J_\theta^n
= &
\mathsf{E}_\theta
[(- n \frac{d \phi}{d \theta}(\theta)
+ \frac{d}{d\theta} \log \overline{P}^3_{\theta}(X_{n+1})
- \frac{d}{d\theta} \log \overline{P}^3_{\theta}(X_1)
+ g^n(X^{n+1}) )^2]
+ J_{\theta}^1   \\
=& 
\| - n \frac{d \phi}{d \theta}(\theta)
+ \frac{d}{d\theta} \log \overline{P}^3_{\theta}(X_{n+1})
- \frac{d}{d\theta} \log \overline{P}^3_{\theta}(X_1)
+ g^n(X^{n+1}) \|_2^2
+ 
J_{\theta}^1 .
\Label{29-2}
\end{align}
Since (\ref{27-15}) and (\ref{27-15b}) yield that
$\frac{1}{n} \| g^n(X^{n+1}) - n \frac{d \phi}{d \theta}(\theta) \|_2^2
\to \frac{d^2 \phi}{d \theta^2}(\theta)$,
we have
\begin{align}
\nonumber
& \frac{1}{\sqrt{n}}\| - n \frac{d \phi}{d \theta}(\theta)
+ \frac{d}{d\theta} \log \overline{P}^3_{\theta}(X_{n+1})
- \frac{d}{d\theta} \log \overline{P}^3_{\theta}(X_1)
+ g^n(X^{n+1}) \|_2  \\
\le &
\frac{1}{\sqrt{n}}( \| g^n(X^{n+1}) - n \frac{d \phi}{d \theta}(\theta) \|_2 
+\|\frac{d}{d\theta} \log \overline{P}^3_{\theta}(X_{n+1})\|_2
+\| \frac{d}{d\theta} \log \overline{P}^3_{\theta}(X_1) \|_2) \nonumber \\
\le &
\frac{1}{\sqrt{n}} \| g^n(X^{n+1}) - n \frac{d \phi}{d \theta}(\theta) \|_2 
+\frac{2}{\sqrt{n}} \max_{x}|\frac{d}{d\theta} \log \overline{P}^3_{\theta}(x)| 
\to  \sqrt{\frac{d^2 \phi}{d \theta^2}(\theta)}.
\Label{29-3}
\end{align}
The combination of (\ref{29-2}) and (\ref{29-3}) yields that 
$\lim_{n\to \infty}\frac{J_\theta^n}{n} \le \frac{d^2 \phi}{d\theta^2}(\theta)$.
Similarly, the opposite inequality can be shown by replacing the role of
(\ref{29-3}) by the following inequality. 
\begin{align*}
& \frac{1}{\sqrt{n}}\| - n \frac{d \phi}{d \theta}(\theta)
+ \frac{d}{d\theta} \log \overline{P}^3_{\theta}(X_{n+1})
- \frac{d}{d\theta} \log \overline{P}^3_{\theta}(X_1)
+ g^n(X^{n+1}) \|_2 \nonumber \\
\ge &
\frac{1}{\sqrt{n}} \| g^n(X^{n+1}) - n \frac{d \phi}{d \theta}(\theta) \|_2 
-\frac{2}{\sqrt{n}} \max_{x}|\frac{d}{d\theta} \log \overline{P}^3_{\theta}(x)| 
\to \sqrt{\frac{d^2 \phi}{d \theta^2}(\theta)}.
\end{align*}
Hence, we obtain (\ref{29-1}).
Since $\frac{d \theta}{d \eta}(\theta)=\frac{d^2 \phi}{d\theta^2}(\theta)^{-1}$,
(\ref{29-1}) implies $\lim_{n\to \infty}\frac{\tilde{J}_\theta^n}{n}=
\frac{d^2 \phi}{d\theta^2}(\theta)^{-1}$.
\end{proofof}

\section{Estimation with multi-parameter case}\Label{s8}
\subsection{Estimation with multi-parameter exponential family: stationary case}\Label{s8-1}
Assume that $W_{\vec{\theta}}$ is a multi-parameter exponential family of transition matrices with $\vec{\theta}=(\theta^1, \ldots,\theta^d)$ with the generator $\{g_j\}$.
Then,
we assume that 
the initial distribution is the stationary distribution
$\overline{P}^1_{\vec{\theta}}$ on $\cX$ of $W_{\vec{\theta}}$
and the random variable $X^{n+1}:= (X_{n+1}, \ldots,X_1)$
is subject to 
$W_{\vec{\theta}}^{\times n} \times \overline{P}^1_{\vec{\theta}}$ with the unknown parameter ${\vec{\theta}}$.
In this subsection, we denote the expectation and the variance under the distribution $W_{\vec{\theta}}^{\times n} \times \overline{P}^1_{\vec{\theta}}$
by $\mathsf{E}_{\vec{\theta}}$ and $\mathsf{V}_{\vec{\theta}}$.

Similar to (\ref{25-7}),
using $\vec{g}^n(X^{n+1}):= [g_j^n(X^{n+1})]_{j}$,
we can show that
\begin{align}
\mathsf{E}_{\vec{\theta}} [\frac{\vec{g}^n(X^{n+1})}{n}] 
=\vec{\eta}(\vec{\theta}),
\Label{25-7c}
\end{align}
which implies that $\vec{g}^n(X^{n+1})$ is an unbiased estimator of the 
expectation parameter $\vec{\eta}(\vec{\theta})$.
We denote the covariance matrix of $\vec{g}^n(X^{n+1})$
by $\mathsf{Cov}_\theta [\vec{g}^n(X^{n+1})]$.
We also denote the covariance matrix of $[
\frac{\partial}{\partial\theta^j} \log \overline{P}^3_{\vec{\theta}}(X)]_j$
by $\hat{\mathsf{Cov}}_\theta$.
\begin{lemma}\Label{L5-2}
The matrix inequalities
\begin{align}
\nonumber 
& n \mathsf{H}_{\vec{\theta}} [\phi]
(1-2\sqrt{\frac{\|
\mathsf{H}_{\vec{\theta}} [\phi]^{-\frac{1}{2}}
\hat{\mathsf{Cov}}_\theta  
\mathsf{H}_{\vec{\theta}} [\phi]^{-\frac{1}{2}}\|
}{ n }})^2
\le \mathsf{Cov}_\theta [\vec{g}^n(X^{n+1})] \\
 \le &
n \mathsf{H}_{\vec{\theta}} [\phi]
(1+2\sqrt{\frac{\|
\mathsf{H}_{\vec{\theta}} [\phi]^{-\frac{1}{2}}
\hat{\mathsf{Cov}}_\theta  
\mathsf{H}_{\vec{\theta}} [\phi]^{-\frac{1}{2}}\|
}{ n }})^2 
\Label{25-5c}
\end{align}
hold,
where the matrix inequality is defined by the positive semi-definiteness.
\end{lemma}
\begin{proof}
First, we fix a real unit vector $\vec{a}=[a_j]_j$.
Applying (\ref{25-5}) to the random variable $\sum_j a_j g_j^n(X^{n+1})$,
we obtain
\begin{align}
\nonumber 
 n \vec{a}^T \mathsf{H}_{\vec{\theta}} [\phi] \vec{a}
(1-2\sqrt{\frac{
\vec{a}^T \hat{\mathsf{Cov}}_\theta  \vec{a}
}{ n
\vec{a}^T \mathsf{H}_{\vec{\theta}} [\phi]\vec{a} }})^2
\le & \vec{a}^T \mathsf{Cov}_\theta [\vec{g}^n(X^{n+1})]\vec{a} \\
 \le &
n \vec{a}^T \mathsf{H}_{\vec{\theta}} [\phi] \vec{a}
(1+2\sqrt{\frac{
\vec{a}^T \hat{\mathsf{Cov}}_\theta  \vec{a}
}{ n
\vec{a}^T \mathsf{H}_{\vec{\theta}} [\phi]\vec{a} }})^2.
\Label{5-11}
\end{align}
Since 
$\frac{
\vec{a}^T \hat{\mathsf{Cov}}_\theta  \vec{a}
}{\vec{a}^T \mathsf{H}_{\vec{\theta}} [\phi]\vec{a} }
\le 
\|
\mathsf{H}_{\vec{\theta}} [\phi]^{-\frac{1}{2}}
\hat{\mathsf{Cov}}_\theta  
\mathsf{H}_{\vec{\theta}} [\phi]^{-\frac{1}{2}}\|$,
(\ref{5-11}) implies (\ref{25-5c}).
\end{proof}

Lemma \ref{L5-2} yields that
\begin{align}
\mathsf{Cov}_{\vec{\theta}} [\frac{\vec{g}^n(X^{n+1})}{n}] 
=
\frac{\mathsf{Cov}_{\vec{\theta}} [\vec{g}^n(X^{n+1})] }{n^2}
= \frac{ \mathsf{H}_{\vec{\theta}}[\phi]}{n} +o(\frac{1}{n}).
\Label{25-8c}
\end{align}

Now, we denote the Fisher information matrix of 
the distribution family $\{ \overline{P}^1_{\vec{\theta}}\}_{\vec{\theta}}$
by $J_{\vec{\theta}}^1$.
The Fisher information matrix $\tilde{J}^{n+1}_{\vec{\eta}(\vec{\theta})}$
for the expectation parameter $\vec{\eta}$ 
of the distribution family $\{ W_{\vec{\theta}}^{\times n} \times \overline{P}^1_{\vec{\theta}}\}_{\vec{\theta}}$ is 
\begin{align*}
& \tilde{J}^{n+1}_{\vec{\eta}(\vec{\theta})}=
([\frac{\partial \eta_i(\vec{\theta})}{\partial \theta_j}]_{i,j}^T)^{-1}
{J}^{n+1}_{\vec{\theta}} 
([\frac{\partial \eta_i(\vec{\theta})}{\partial \theta_j}]_{i,j})^{-1}
=
\mathsf{H}_{\vec{\theta}}[\phi]^{-1}
(n 
\mathsf{H}_{\vec{\theta}}[\phi]
+ J_{\vec{\theta}}^1)
\mathsf{H}_{\vec{\theta}}[\phi]^{-1} \\
=&
\mathsf{H}_{\vec{\theta}}[\phi]^{-\frac{1}{2}}
(n 
I+ 
\mathsf{H}_{\vec{\theta}}[\phi]^{-\frac{1}{2}}
J_{\vec{\theta}}^1
\mathsf{H}_{\vec{\theta}}[\phi]^{-\frac{1}{2}}
)
\mathsf{H}_{\vec{\theta}}[\phi]^{-\frac{1}{2}}.
\end{align*}
That is,
the lower bound of the variance of the unbiased estimator 
given by Cram\'{e}r-Rao inequality is
$
\frac{1}{n}
\mathsf{H}_{\vec{\theta}}[\phi]^{\frac{1}{2}}
(1 
I+ \frac{1}{n}
\mathsf{H}_{\vec{\theta}}[\phi]^{-\frac{1}{2}}
J_{\vec{\theta}}^1
\mathsf{H}_{\vec{\theta}}[\phi]^{-\frac{1}{2}}
)^{-1}
\mathsf{H}_{\vec{\theta}}[\phi]^{\frac{1}{2}}$,
i.e., the Cram\'{e}r-Rao inequality is given as
\begin{align}
\nonumber
\mathsf{Cov}_{\vec{\theta}} [\frac{\vec{g}^n(X^{n+1})}{n}] 
\ge &
\frac{1}{n}
\mathsf{H}_{\vec{\theta}}[\phi]^{\frac{1}{2}}
(I+ \frac{1}{n}
\mathsf{H}_{\vec{\theta}}[\phi]^{-\frac{1}{2}}
J_{\vec{\theta}}^1
\mathsf{H}_{\vec{\theta}}[\phi]^{-\frac{1}{2}}
)^{-1}
\mathsf{H}_{\vec{\theta}}[\phi]^{\frac{1}{2}} \\
=&
\frac{1}{n}
\mathsf{H}_{\vec{\theta}}[\phi]+
O(\frac{1}{n^2}). 
\Label{25-6c}
\end{align}
The relation (\ref{25-8c}) shows that 
the unbiased estimator $\frac{\vec{g}^n(X^{n+1})}{n}$ realizes the optimal performance with the order $\frac{1}{n}$.

Therefore, we obtain an asymptotically efficient estimator for the expectation parameter.
To estimate the natural parameter, 
we need to solve the equation
\begin{align}
\eta_j= \frac{\partial \phi}{\partial \theta^j}(\vec{\theta})
\end{align}
for $\vec{\theta}$.
Since the function $\phi(\vec{\theta})$ is strictly convex,
$\vec{\theta}(\vec{\eta})$ can be derived by 
the maximization of the concave function as
\begin{align}
\argmax_{\vec{\theta}}
\vec{\eta}\cdot\vec{\theta}-\phi (\vec{\theta}).
\end{align}
The calculation complexity does not depend on the number $n$ of data. 
Hence, when the number $d$ of parameters is not so large, the natural parameter can be estimated efficiently even with a large number $n$ of data.

However, the conventional algorithm for the maximization of the concave function \cite{BV} requires the calculation of the derivative.
Since the convex function $\phi(\vec{\theta})$
is given as the logarithm of the Perron-Frobenius 
eigenvalue of the matrix $\overline{W}_\theta$,
the calculation of the derivative is not so easy.
To overcome this kind of difficulty, 
we can employ derivative-free optimization algorithms \cite{CSV,MKK} represented by Nelder-Mead method \cite{NM}.
A derivative-free optimization algorithm
maximizes a concave function 
without calculating the derivative only with calculating the outcomes with several inputs.
In particular, it is expected that such an algorithm enables us to numerically derive $\vec{\theta}(\vec{\eta})$ for a given $\vec{\eta}$.

\subsection{Estimation with multi-parameter exponential family: non-stationary case}\Label{s8-2}
Next, similar to Section \ref{s7},
we consider the non-stationary case and assume that the transition matrices $W_{\vec{\theta}}$ are ergodic as well as irreducible.
Then, we fix an arbitrary initial distributions $P_{\vec{\theta}}$ on $\cX$ such that
the distribution $P_{\vec{\theta}}$ is smoothly parameterized by the natural parameter $\vec{\theta}$.
This assumption contains 
the special case when the distribution $P_{\vec{\theta}}$ does not depend on the parameter $\vec{\theta}$.

In this subsection, we denote the expectation, the variance, and the covariance matrix under the distribution $W_{\vec{\theta}}^{\times n} \times P_{\vec{\theta}}$
by $\mathsf{E}_{\vec{\theta}}$, 
$\mathsf{V}_{\vec{\theta}}$, and $\mathsf{Cov}_{\vec{\theta}}$.
Then, we employ the random variable $\vec{g}^n(X^{n+1}):=({g}_j^n(X^{n+1}))$.
When we use the random variable $\vec{g}^n (X^{n+1})/n$ as an estimator of the parameter $\vec{\theta}$, the error is measured by
the mean square error matrix:
\begin{align*}
\mathsf{MSE}_\theta [\frac{\vec{g}^n (X^{n+1})}{n}]_{i,j}:=
\mathsf{E}_\theta [(\frac{{g}_i^n (X^{n+1})}{n}-\eta_i(\vec{\theta}) )
(\frac{{g}_j^n (X^{n+1})}{n}-\eta_j(\vec{\theta})) ].
\end{align*}
Similar to (\ref{27-15}), we can show that
\begin{align}
\lim_{n\to \infty} \mathsf{E}_{\vec{\theta}} [\frac{\vec{g}^n(X^{n+1})}{n}] 
=&
\vec{\eta}(\vec{\theta}) =
[\frac{ \partial \phi}{\partial  \theta^j}(\vec{\theta})]_j 
\Label{27-15c}.
\end{align}
For any vector $\vec{c}=(c_i)$, 
the application of (\ref{27-15b}) to $\theta= \vec{c} \cdot \vec{\theta}$
implies that
\begin{align*}
\lim_{n\to \infty} 
n 
\vec{c}^T
\mathsf{MSE}_{\vec{\theta}} [\frac{\vec{g}^n(X^{n+1})}{n}]
\vec{c}
= 
\lim_{n\to \infty} 
n 
\vec{c}^T
\mathsf{Cov}_{\vec{\theta}} [\frac{\vec{g}^n(X^{n+1})}{n}] 
\vec{c}
= 
\vec{c}^T
\mathsf{H}_{\vec{\theta}} [\phi]
\vec{c},
\end{align*}
which implies the following theorem.
\begin{theorem}
\begin{align}
\lim_{n\to \infty} 
n 
\mathsf{MSE}_{\vec{\theta}} [\frac{\vec{g}^n(X^{n+1})}{n}]
= 
\lim_{n\to \infty} 
n 
\mathsf{Cov}_{\vec{\theta}} [\frac{\vec{g}^n(X^{n+1})}{n}] 
= 
\mathsf{H}_{\vec{\theta}} [\phi].
\Label{27-15bc}
\end{align}
\end{theorem}
The relation (\ref{27-15c}) shows that the estimator $\frac{\vec{g}^n(X^{n+1})}{n}$
is asymptotically unbiased for the expectation parameter $\vec{\eta}$.
The above theorem implies that the mean square error is
$\frac{1}{n}\mathsf{H}_{\vec{\theta}} [\phi]
 +o(\frac{1}{n})$.

Next, for the family $\{W^{\times n}_{\vec{\theta}} \times P_{\vec{\theta}}\}_{\vec{\theta}}$, 
we consider the Fisher information matrix $J_{\vec{\theta}}^n$ for the natural parameter $\vec{\theta}$
and the Fisher information matrix $\tilde{J}_{\vec{\theta}}^n$ for the expectation parameter $\vec{\eta}$.
\begin{lemma}\Label{L29-1c}
The limit of the Fisher information matrix $J_{\vec{\theta}}^n$ for the natural parameter $\vec{\theta}$
is characterized as
$\lim_{n\to \infty}\frac{J_{\vec{\theta}}^n}{n}=
\mathsf{H}_{\vec{\theta}} [\phi]$.
Hence, the limit of the Fisher information matrix $\tilde{J}_{\vec{\theta}}^n$ 
for the expectation parameter $\vec{\eta}$ is characterized as
$\lim_{n\to \infty}\frac{\tilde{J}_{\vec{\theta}}^n}{n}=
\mathsf{H}_{\vec{\theta}} [\phi]^{-1}$.
\end{lemma}

\begin{proof}
We fix a real unit vector $\vec{a}=[a_j]_j$.
The application of the relation (\ref{29-1}) 
to $\theta= \vec{c} \cdot \vec{\theta}$
yields that
$\lim_{n\to \infty}\frac{ \vec{a}^T J_{\vec{\theta}}^n \vec{a} }{n}=
\vec{a}^T \mathsf{H}_{\vec{\theta}} [\phi] \vec{a}$,
which implies $\lim_{n\to \infty}\frac{J_{\vec{\theta}}^n}{n}=
\mathsf{H}_{\vec{\theta}} [\phi]$.
Since $[\frac{\partial \eta_i(\vec{\theta})}{\partial \theta_j}]_{i,j}$ is 
$\mathsf{H}_{\vec{\theta}} [\phi]$, we obtain
$\lim_{n\to \infty}\frac{\tilde{J}_{\vec{\theta}}^n}{n}=
\mathsf{H}_{\vec{\theta}} [\phi]^{-1}$.
\end{proof}

Lemma \ref{L29-1c} 
implies that
the lower bound of the Cram\'{e}r-Rao inequality is 
$\frac{1}{n}
\mathsf{H}_{\vec{\theta}} [\phi]
+o(\frac{1}{n})$.
Therefore, 
the estimator $\frac{\vec{g}^n(X^{n+1})}{n}$ attains the lower bound by the Cram\'{e}r-Rao inequality
with the order $\frac{1}{n}$.
That is, the estimator $\frac{\vec{g}^n(X^{n+1})}{n}$ is asymptotically efficient.

Similar to the one-parameter case,
we can show that
the random variable 
$ \sqrt{n}(
\frac{\vec{g}^n(X^{n+1})}{n} - \vec{\eta}(\vec{\theta}))$
converges to
the Gaussian distribution with the covariance matrix
$\mathsf{H}_{\vec{\theta}}[\phi]$.

\subsection{Estimation with multi-parameter curved exponential family}\Label{s8-3}
Next, we proceed to estimation with multi-parameter curved exponential family.
A $d'$-parameter subset 
$\tilde{{\cal E}}=\{W_{\vec{\theta}_{\rm CRV}(\vec{\xi})}\}_{\vec{\xi}}$
of an exponential family ${\cal E}=\{W_{\vec{\theta}}\}_{\vec{\theta}}$
of transition matrices
is called a {\it curved exponential family} of transition matrices.
For example, a mixture family defined in Subsection \ref{s4-5-2} is also a curved exponential family.
As explained in Example \ref{1-3-8},
the set of all positive transition matrices on a finite-size system
forms an exponential family.
Hence, any smooth subfamily of transition matrices on 
a finite-size system forms a curved exponential family.
Then, we define the Fisher information matrix
$\tilde{\mathsf{H}}_{\vec{\xi}}$ as the metric of the submanifold.
Assume that the Jacobian matrix 
$A :=[\frac{\partial \eta_i}{\partial \xi_j}|_{\vec{\xi}=\vec{\xi}_o}]_{i,j}$
has the rank $d'$.
When the potential function of the exponential family is $\phi$,
the Fisher information matrix is written as
$\tilde{\mathsf{H}}_{\vec{\xi}_o}
=
A^T \mathsf{H}_{\vec{\theta}_{\rm CRV}(\vec{\xi}_o)}[\phi]^{-1} A$
because the Fisher information matrix for the expectation parameter $\eta$
at $\vec{\theta}_{\rm CRV}(\vec{\xi}_o)$
is $\mathsf{H}_{\vec{\theta}_{\rm CRV}(\vec{\xi}_o)}^{-1}$.

\begin{figure}[htbp]
\begin{center}
\scalebox{0.7}{\includegraphics[scale=0.5]{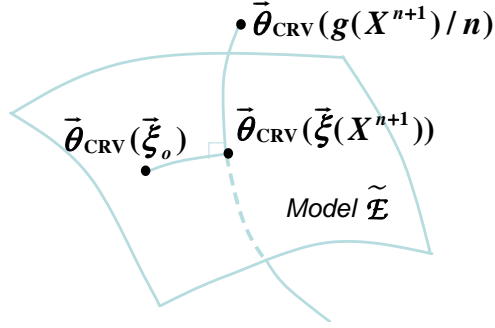}}
\end{center}
\caption{Estimator for the curved exponential family.}
\Label{f1}
\end{figure}%

In the following, we assume that the exponential family ${\cal E}$ is generated by $g_j$.
Given $n+1$ observations $X^{n+1}$, as Fig. \ref{f1},
we define the estimator 
$\vec{\xi}^n(X^{n+1})
:=\argmin_{\xi} D(W_{\vec{\theta}_{\rm CRV}(\vec{g}^n(X^{n+1})/n)}
\|W_{\vec{\theta}_{\rm CRV}(\vec{\xi})})$ for 
the curved exponential family $\tilde{{\cal E}}$.
Then, similar to the case of a curved exponential family of probability distributions 
\cite[Section 4.4]{AN},
we can show that
the estimator $\vec{\xi}^n(X^{n+1})$ is asymptotically efficient.
That is, the mean square error matrix is asymptotically approximated to 
$\frac{1}{n} \tilde{\mathsf{H}}_{\vec{\xi}} [\phi]^{-1}+ o(\frac{1}{n})$ as follows.

\begin{theorem}
The random variable $\vec{\xi}^n(X^{n+1})- \vec{\xi}_o$
asymptotically obeys
the Gaussian distribution with the covariance matrix 
$\frac{1}{n}
\tilde{\mathsf{H}}_{\vec{\xi}_o}[\phi]^{-1}$.
Then, the mean square error matrix of
our estimator $\vec{\xi}^n(X^{n+1})$
is asymptotically approximated to 
$\frac{1}{n} \tilde{\mathsf{H}}_{\vec{\xi}} [\phi]^{-1}+ o(\frac{1}{n})$.
\end{theorem}

\begin{proof}
The random variable $
\vec{g}^n(X^{n+1})/n - \vec{\eta}_o$
asymptotically obeys the Gaussian distribution with the covariance matrix 
$\frac{1}{n}\mathsf{H}_{\vec{\theta}_{\rm CRV}(\vec{\xi}_o)}[\phi]$,
where 
$\vec{\eta}_o:= \vec{\eta}(\vec{\theta}_{\rm CRV}(\vec{\xi}_o)))$.
Since the neighborhood of $\vec{\xi}^n(X^{n+1})$ in $\tilde{{\cal E}}$
can be approximated to the tangent space at the true point $\vec{\xi}_o$, 
due to Corollary \ref{C15-1},
the point $\vec{\theta}_{\rm CRV}(\vec{\xi}^n(X^{n+1}))$
can be approximately regarded as 
the projection to the tangent space at $\vec{\xi}_o$
from the observed point 
$\vec{\theta}_{\rm CRV}(\vec{g}^n(X^{n+1})/n)$.

To see the asymptotic variance of the random variable 
$\vec{\xi}^n(X^{n+1})- \vec{\xi}_o$,
we choose a $d \times d'$ matrix $B_1$ and
a $d \times (d-d')$ matrix $B_2$ such that
the $d \times d$ matrix $B=(B_1,B_2)$ satisfies that
\begin{align}
B^T \mathsf{H}_{\vec{\theta}_{\rm CRV}(\vec{\xi}_o)}[\phi]^{-1} B
=I \quad  \hbox{ and } \quad
B_2^T \mathsf{H}_{\vec{\theta}_{\rm CRV}(\vec{\xi}_o)}[\phi]^{-1} A
=0.
\Label{11-18-1} 
\end{align}
Then, 
$B_1^T \mathsf{H}_{\vec{\theta}_{\rm CRV}(\vec{\xi}_o)}[\phi]^{-1} A$
is invertible.
So, 
$A^T \mathsf{H}_{\vec{\theta}_{\rm CRV}(\vec{\xi}_o)}[\phi]^{-1} A
=$\par\noindent $
(B_1^T \mathsf{H}_{\vec{\theta}_{\rm CRV}(\vec{\xi}_o)}[\phi]^{-1} A)^T
B_1^T \mathsf{H}_{\vec{\theta}_{\rm CRV}(\vec{\xi}_o)}[\phi]^{-1} B_1
(B_1^T \mathsf{H}_{\vec{\theta}_{\rm CRV}(\vec{\xi}_o)}[\phi]^{-1} A)
$\par\noindent $=
(B_1^T \mathsf{H}_{\vec{\theta}_{\rm CRV}(\vec{\xi}_o)}[\phi]^{-1} A)^T
(B_1^T \mathsf{H}_{\vec{\theta}_{\rm CRV}(\vec{\xi}_o)}[\phi]^{-1} A)$.
Now, we introduce the new parameter $\vec{\tau}(\vec{\eta}):= B^{-1}\vec{\eta}$
under which, the metric is given as Cartesian inner product.
Hence, the covariance matrix of the estimator $B^{-1}(\vec{g}^n(X^{n+1})/n)$
for the parameter $\vec{\tau}(\vec{\eta})$ is the matrix $\frac{1}{n}I$.

We denote  the vector $(\tau_1, \ldots, \tau_{d'})^T$
by $\vec{\tau}'(\vec{\eta})$.
Since the parameter $\vec{\xi}$ is approximately identified
with the element of the tangent space,
we have 
$A (\vec{\xi}-\vec{\xi}_o) =\vec{\eta}-\vec{\eta}_o=
B (\vec{\tau}(\vec{\eta})-\vec{\tau}(\vec{\eta}_o))$.
Hence, 
\eqref{11-18-1} implies that
\begin{align*}
& B_1^T  \mathsf{H}_{\vec{\theta}_{\rm CRV}(\vec{\xi}_o)}[\phi]^{-1} A (\vec{\xi}-\vec{\xi}_o) 
=
B_1^T  \mathsf{H}_{\vec{\theta}_{\rm CRV}(\vec{\xi}_o)}[\phi]^{-1} B 
(\vec{\tau}(\vec{\eta}(\vec{\xi}))-\vec{\tau}(\vec{\eta}_o)) \\
= & B_1^T  \mathsf{H}_{\vec{\theta}_{\rm CRV}(\vec{\xi}_o)}[\phi]^{-1} 
B_1 
(\vec{\tau}'(\vec{\eta}(\vec{\xi}))-\vec{\tau}'(\vec{\eta}_o))
=
\vec{\tau}'(\vec{\eta}(\vec{\xi}))-\vec{\tau}'(\vec{\eta}_o).
\end{align*}
Thus,
\begin{align*}
\vec{\xi}-\vec{\xi}_o
=
(B_1^T  \mathsf{H}_{\vec{\theta}_{\rm CRV}(\vec{\xi}_o)}[\phi]^{-1} A)^{-1}
(\vec{\tau}'(\vec{\eta}(\vec{\xi}))-\vec{\tau}'(\vec{\eta}_o)).
\end{align*}
In this approximation, 
our estimator $\vec{\xi}^n(X^{n+1})$ for $\vec{\xi}$ is characterized as\par\noindent
$(B_1^T  \mathsf{H}_{\vec{\theta}_{\rm CRV}(\vec{\xi}_o)}[\phi]^{-1} A)^{-1}
B^{-1}(\vec{g}^n(X^{n+1})/n-\vec{\eta}_o)+\vec{\xi}_o$.
Thus, the covariance matrix of our estimator is
\begin{align*}
& (B_1^T \mathsf{H}_{\vec{\theta}_{\rm CRV}(\vec{\xi}_o)}[\phi]^{-1} A)^{-1}
\frac{I}{n} ((B_1^T \mathsf{H}_{\vec{\theta}_{\rm CRV}(\vec{\xi}_o)}[\phi]^{-1} A)^T)^{-1} \\
=&
\frac{1}{n}((B_1^T \mathsf{H}_{\vec{\theta}_{\rm CRV}(\vec{\xi}_o)}[\phi]^{-1} A)^T
(B_1^T \mathsf{H}_{\vec{\theta}_{\rm CRV}(\vec{\xi}_o)}[\phi]^{-1} A) )^{-1} \\
=&
\frac{1}{n}(A^T \mathsf{H}_{\vec{\theta}_{\rm CRV}(\vec{\xi}_o)}[\phi]^{-1} A)^{-1}
=\frac{1}{n}
\tilde{\mathsf{H}}_{\vec{\xi}_o}[\phi]^{-1}.
\end{align*}
That is, the random variable $\vec{\xi}^n(X^{n+1})- \vec{\xi}_o$
asymptotically obeys
the Gaussian distribution with the covariance matrix 
$\frac{1}{n}
\tilde{\mathsf{H}}_{\vec{\xi}_o}[\phi]^{-1}$.
Therefore, the mean square error matrix of
our estimator $\vec{\xi}^n(X^{n+1})$
is asymptotically approximated to 
$\frac{1}{n} \tilde{\mathsf{H}}_{\vec{\xi}} [\phi]^{-1}+ o(\frac{1}{n})$.
\end{proof}

\begin{remark}\Label{rem8-4}
The papers \cite{Feigin,KS,Hudson,Bhat,Stefanov,Kuchler-Sorensen,Sorensen} showed that
the maximum likelihood estimator (MLE) is asymptotically efficient in the exponential family with their definition (\ref{1-3}).
Since the definition (\ref{1-3}) is different from ours (\ref{5-6}),
the results in this section are different from theirs.
Further, since our asymptotically efficient estimator is given as the sample mean of $g$, 
the required calculation amount is smaller than theirs.
Even in the case of a curved exponential family,
the Pythagorean theorem (\ref{5-1}) enables us to calculate our asymptotically efficient estimator
with small amount of calculation.
However, their MLE does not have so simple form
because their exponential family does not have such a geometrical structure, e.g., expectation parameter and the Pythagorean theorem, etc.
Hence, it requires large calculation amounts.

Indeed, when the matrix entries of the transition matrix is to be estimated,
the literature \cite{CS} showed that 
the sample mean is the same as the maximum likelihood estimator.
However, this fact holds only for such a specific parameter,
and cannot be applied to the parameter estimation of our exponential family, in general.
Our method can be applied to any parameter of an exponential family in our sense.
\end{remark}

\subsection{Implementation of our estimator for curved exponential family}\Label{s8-4}
In this subsection, we consider how to calculate our estimator $\vec{\xi}^n(X^{n+1})$.
This calculation depends on the type of parametrization of 
the transition matrix $W_{\vec{\theta}_{\rm CRV}(\vec{\xi})}$.
We can consider two cases as follows.
\begin{description}
\item[(1)]
The entries of the transition matrix $W_{\vec{\theta}_{\rm CRV}(\vec{\xi})}$
are calculated directly from $\vec{\xi}$ with small calculation complexity.

\item[(2)]
The entries of the transition matrix $W_{\vec{\theta}_{\rm CRV}(\vec{\xi})}$
are calculated by \eqref{12-26-1}
via the calculation of $\vec{\theta}_{\rm CRV}(\vec{\xi})$.
In this case, the calculation of these entries
has large calculation complexity.

\end{description}

For example, Example \ref{1-3-8c} belongs to Case (1) because $W_{\vec{\eta}}$ is directly calculated from the parameter $\vec{\eta}$.

In the calculation of the estimator $\vec{\xi}^n(X^{n+1})$,
first, we obtain the estimate $\vec{\eta}'$ of the larger exponential family ${\cal E}$ with the expectation parameter.  
Then, 
we calculate its natural parameter $\vec{\theta}'$ 
by the method given in the end of Subsection \ref{s8-1}.
The following steps depend on the above case.
In Case (1), 
we can implement the minimization by employing the final expression 
in \eqref{eq:markov-divergence-cdf-relation} with small calculation complexity
due to the following reason.
The final expression in \eqref{eq:markov-divergence-cdf-relation}
needs only the entries of the transition matrices
$W_{\vec{\theta}'} $ and $W_{\vec{\theta}_{\rm CRV}(\vec{\xi})}$
and the Perron-Frobenius eigenvector of $W_{\vec{\theta}'}$.
In this case, it is enough to calculate the Perron-Frobenius eigenvalue
the Perron-Frobenius eigenvector of $W_{\vec{\theta}'}$
only at the first step.
At each step of the minimization,
we do not have any difficult calculation.
Therefore, 
the final expression in \eqref{eq:markov-divergence-cdf-relation} brings us 
an easy implementation of the minimization in Case (1).

However, in Case (2), it is better to employ \eqref{1-1} instead of the final expression in \eqref{eq:markov-divergence-cdf-relation} due to the following reason.
When the final expression in \eqref{eq:markov-divergence-cdf-relation} is employed,
the calculation of the transition matrix 
$W_{\vec{\theta}_{\rm CRV}(\vec{\xi})}$ 
requires the calculations of 
the Perron-Frobenius eigenvalue
and the Perron-Frobenius eigenvector 
of the matrix given in \eqref{5-6} as in \eqref{12-26-1}.
To calculate the RHS of \eqref{1-1}, we need to calculate
the partial derivative $\frac{\partial \phi}{\partial \theta^j}(\vec{\theta}')$
and the Perron-Frobenius eigenvalues
 $\phi(\vec{\theta}')$ and $\phi(\vec{\theta}_{\rm CRV}(\vec{\xi}))$.
Fortunately,
the partial derivative $\frac{\partial \phi}{\partial \theta^j}(\vec{\theta}')$
coincides with the expectation parameter $\vec{\eta}'$, which is firstly obtained.
Also, it is enough to calculate the Perron-Frobenius eigenvalue
 $\phi(\vec{\theta}')$ only once.
Hence, at each step of the minimization,
we need to calculate only the Perron-Frobenius eigenvalue
$\phi(\vec{\theta}_{\rm CRV}(\vec{\xi}))$, i.e., we do not need to calculate 
the Perron-Frobenius vector. 
Therefore, 
\eqref{1-1} requires less calculation complexity than the final expression in \eqref{eq:markov-divergence-cdf-relation} in Case (2).

\section{Conclusion}\Label{s9}
We have revisited the information geometrical structure 
(the exponential family, the natural parameter, the expectation parameter, 
relative entropy, relative R\'{e}nyi entropy, Fisher information matrix,
and 
the Pythagorean theorem)
of transition matrices
by using the convex function 
$\phi(\theta)$
defined by the Perron-Frobenius eigenvalue
of the matrix $\overline{W}_{\vec{\theta}}$ defined by (\ref{5-6}).
Then, we have shown that the sample mean of the generating function
is an asymptotically efficient estimator for the expectation parameters
in the exponential family of transition matrices.
Combining this property and the Pythagorean theorem, 
we have given an asymptotically efficient estimator for a curved exponential family of transition matrices.
As a consequence, 
we have characterized the asymptotic variance 
of the sample mean in the Markovian chain by using the second derivative of 
the convex function $\phi(\theta)$.

In this paper, we have assumed that 
our system consists of finite elements.
Indeed, the existing papers \cite{C7,C8,C9,C10,C11} reported 
several difficulties to evaluate the variance
of the sample mean in the continuous probability space 
even with the discrete time Markov chain.
So, it is remained to extend the obtained results to the continuous case. 
However, this assumption is assumed only for describing the conditional distribution by a matrix.
We do not use the finiteness of the cardinality of the probability space explicitly.
Therefore, it seems that there is no essential obstacle for extension to the continuous case under a proper regularity condition. 
This extension will enable us to handle several 
Gaussian Markovian chains in a simple way. 
Further, the obtained version of the Pythagorean theorem
will be helpful for the hierarchy of exponential families of transition matrices.
For an example, 
a hierarchy of exponential families can be constructed 
by changing the degree of Markovian chain,
it might be interesting to investigate this example.

\section*{Acknowledgment}
The authors are grateful for Dr. Wataru Kumagai to informing the references \cite{CSV,NM,MKK}. 
MH is partially supported by a MEXT Grant-in-Aid for Scientific Research (A) No. 23246071. 
MH is also partially supported by the National Institute of Information and Communication Technology (NICT), Japan.
SW is partially supported by JSPS Postdoctoral Fellowships for Research Abroad.
The Centre for Quantum Technologies is funded by the Singapore Ministry of Education and the National Research Foundation as part of the Research Centres of Excellence programme.

\appendix

\section{Relation with existing results}\Label{as1}
As mentioned in Introduction, some of results in this paper 
for relative entropy and exponential family
have been already stated
in \cite{HN} (without detailed proof) and we restate those results and give proofs to keep the paper self-contained.
For deeper understanding, we summarize the relation with those papers in this appendix.

Our definition (\ref{1-10}) for the relative entropy $D(W\|V)$ has the following relation with those by \cite{N,NK,HN}.
Natarajan \cite{N} and Nakagawa and Kanaya \cite{NK}
defined the relative entropy $D(W\|V)$
by the final term of (\ref{eq:markov-divergence-cdf-relation}).
However, Nagaoka \cite{HN} defined the relative entropy $D(W\|V)$ by (\ref{1-1})
and showed the equivalence with the final term of (\ref{eq:markov-divergence-cdf-relation}).
If we consider only the relative entropy $D(W\|V)$,
the definition by the final term of (\ref{eq:markov-divergence-cdf-relation}) is natural.
However, the relative R\'{e}nyi entropy $D_{1+s}(W\|V)$ cannot define in the same way.
Hence, in order to treat the relative entropy $D(W\|V)$ and
the relative R\'{e}nyi entropy $D_{1+s}(W\|V)$ in a unified way,
we adopt the definition (\ref{1-10}) for the relative entropy $D(W\|V)$ 
instead of the final term of (\ref{eq:markov-divergence-cdf-relation}).
Our definition clarifies the relation between 
the relative entropy $D(W\|V)$ and
the relative R\'{e}nyi entropy $D_{1+s}(W\|V)$,
which is helpful when we apply these quantities to 
simple hypothesis testing \cite{HW14-2},
random number generation, data compression,
and channel coding \cite{W-H2}
in Markov chain.

Next, we address the convexity of the function $\phi(\vec{\theta})$.
Nakagawa and Kanaya \cite[Section III]{NK} and Nagaoka \cite{HN}
showed the convexity $\phi(\vec{\theta})$ in their respective cases.
Nagaoka \cite{HN} also showed the equivalence between (1) and (5) in Lemma \ref{L1-14-2}.
However, they did not clearly consider the relation with the other conditions in Lemma \ref{L1-14-2}.
In fact, these equivalence relations are essential 
for the condition of a generator of an exponential family
and also for applications to 
finite-length evaluations of the tail probability, 
the error probability in simple hypothesis testing \cite{HW14-2}, 
source coding, channel coding, and random number generation \cite{W-H2} in Markov chain.

Now, we proceed to the definition of an exponential family for transition matrices.
Our logical order of arguments in this definition is different from that by Nagaoka \cite{HN} and Nakagawa and Kanaya \cite{NK}.
We firstly define the potential function $\phi(\vec{\theta})$ from 
a given transition matrix $W$ and a given generator $\{g_j\}$ 
Then, we give the parametric transition matrices although their papers \cite{HN,NK} gave the parametric transition matrices firstly.
The potential function $\phi(\vec{\theta})$ 
for a transition matrix $W$ and a generator $\{g_j\}$ 
produces several information quantities, 
which play the central roles 
when we apply the exponential family for transition matrices
to finite-length evaluations of the tail probability and
the above applications \cite{HW14-2,W-H2} in Markov chain.
To characterize these information quantities,
we employ an exponential family of transition matrices.
So, our logical order adapts such an application.
Further, this paper introduces a mixture family while
the existing papers \cite{HN,NK} did not define a mixture family. 

Indeed, Kontoyiannis and Meyn \cite[(11)]{CLT2} gave 
a one-parameter family of 
transition matrices with the same logical order.
However, they did not use the terminology ``exponential family'' 
and did not show the convexity of the potential function $\phi(\vec{\theta})$.
Ito and Amari \cite{HA} discussed the geometrical structure of an exponential family of transition matrices
only for ${\cal W}_{{\cal X}}$ in the same definition as ours.
However, they did not treat this set as an exponential family of transition matrices.

Our formula (\ref{5-1}) in Pythagorean theorem (Proposition \ref{P1-15})
has the following relation with Nakagawa and Kanaya \cite{NK}.
Nakagawa and Kanaya \cite[Lemma 5]{NK} showed 
(\ref{5-1}) with $k=1$. 
Hence, our relation (\ref{5-1}) can be regarded as a generalization of 
Nakagawa and Kanaya \cite[Lemma 5]{NK}.
Indeed, the motivation of Nakagawa and Kanaya \cite[Lemma 5]{NK}
is related to the exponent of simple hypothesis testing.
That is, their purpose is to show 
the relation
\begin{align}
\min_{W:D(W\|W_1) \le r} D(W\|W_0)
=
\min_{\theta:D(W_{\theta}\|W_1) \le r} D(W_{\theta}\|W_0).\Label{5-9}
\end{align}
However, 
the multi-parametric extension 
(\ref{5-1}) is essential for estimation in a curved exponential family,
which is discussed in Subsection \ref{s8-3}.

\section{Set of positive bi-stochastic matrices}\Label{as2}
To discuss Example \ref{1-3-8c} in the detail, 
we investigate the set of bi-stochastic matrices on 
${\cal X}=\{0,1, \ldots, m\}$.
First, we divide the linear space of $(m+1)\times (m+1)$ matrices into two linear spaces:
\begin{align}
{\cal A}&:=\{ (v_x+w_y)_{x,y}| (v_x)_x,(w_x)_x \in \mathbb{R}^{m+1} \} \\
{\cal B}&:=\left\{ (a_{x,y})_{x,y} \left| 
\sum_{x'=0}^{m} a_{x',y}=\sum_{y'=0}^{m} a_{x,y'}=0 \hbox{ for }
x,y=0, 1,\ldots, m \right.\right\} .
\end{align}
In the following, any two-input function $g(x,x')$ is regarded as an $(m+1)\times (m+1)$ matrix.
For an arbitrary non-identical permutation $\sigma \in S_{{\cal X}}$,
the function $\hat{g}_\sigma$ belongs to ${\cal B}$.
The function $g_j$ belongs to ${\cal A}$.
Also, when a function $h$ satisfies $h(x,y)=c+v_x- v_y$ with a constant $c$ and a vector $(v_x)\in \mathbb{R}^{m+1}$,
the function $h$ belongs to ${\cal A}$.
Any non-zero linear combination of $\{g_{j}\}_{j=1}^m$
cannot be written by the above function $h$.
Thus, to show the linear independence of the set of functions 
$\{g_{j}\}_{j=1}^m \cup \{\hat{g}_\sigma\}_{\sigma \in T \cup H}$,
it is enough to show the following lemma.

\begin{lemma}\Label{L1-11-2}
The set $\{\hat{g}_\sigma\}_{\sigma \in T \cup H}$
is linearly independent 
in the linear space ${\cal B}$.
\end{lemma}

The number of elements of the set
$\{g_{j}\}_{j=1}^m \cup \{\hat{g}_\sigma\}_{\sigma \in T \cup H}$
is $m^2$, which equals the dimension of ${\cal B}$.
So, the set $\{g_{j}\}_{j=1}^m \cup \{\hat{g}_\sigma\}_{\sigma \in T \cup H}$
spans the linear space ${\cal B}$.
For any bi-stochastic matrix $W$, we have $W-W_{id} \in {\cal B}$. 
Hence, $W-W_{id}$ can be written as 
a linear combination of $\{\hat{g}_\sigma\}_{\sigma \in T \cup H}$,
i.e., $\sum_{\sigma \in T \cup H} \eta_\sigma \hat{g}_\sigma$.
Therefore, 
$W=W_{id}+ \sum_{\sigma \in T \cup H} \eta_\sigma \hat{g}_\sigma
=W_{\vec{\eta}}$.

\begin{proofof}{Lemma \ref{L1-11-2}}
Now, we prepare notations.
For a two-input function $g$,
we define the symmetric matrix
$S[g]_{x,x'}:= g(x,x') +g(x',x)$
and the anti-symmetric matrix
$A[g]_{x,x'}:= g(x,x') -g(x',x)$.

Due to the constraint for ${\cal B}$,
the diagonal entries of an element of $S({\cal B})$ are determined by other entries.
Fixed $0\le i'<j'\le m$, only the matrix $S[\hat{g}_{(i',j')}]$ has a non-zero $(i',j')$-th entry
among the set $\{S[\hat{g}_{(i,j)}]\}_{(i,j)\in T}$.
Hence, the set $\{S[\hat{g}_{(i,j)}]\}_{(i,j)\in T}$ 
is linearly independent in the linear space $S[{\cal B}]$.

Due to the constraint for ${\cal B}$,
the $(0,i)$-th entry and $(i,0)$-th entry of an element of $A({\cal B})$ are determined by other entries.
Fixed $0< i'<j'\le m$,
only the matrix $A[\hat{g}_{(0,i',j')}]$
has a non-zero $(i',j')$-th entry among the set $\{A[\hat{g}_{(0,i,j)}]\}_{(0,i,j)\in H}$. 
Hence, the set $\{A[\hat{g}_{(0,i,j)}]\}_{(0,i,j)\in H}$ is linearly independent
in the linear space $A[{\cal B}]$.
Therefore, 
the set $\{\hat{g}_{(0,i,j)}\}_{(0,i,j)\in H}$ is linearly independent
in the linear space ${\cal B}$.
Since $A[\hat{g}_{(i,j)}]=0$ for $(i,j)\in T$, 
the set $\{\hat{g}_{\sigma}\}_{\sigma\in T\cup H}$ is linearly independent
in the linear space ${\cal B}$.
\end{proofof}

\section{Proofs of Lemmas \ref{L11} and \ref{L11-2}}\Label{as3}
The Fisher information $J_{\theta}^2$ can be written as
\begin{align}
\nonumber 
& J_{\theta}^2
= \sum_{x,x'} W_{\theta} \times \overline{P}^1_{\theta}(x,x') 
[-\frac{d^2}{d\theta^2}
\log W_{\theta} (x|x')  \overline{P}^1_{\theta}(x')]  \\
= &\sum_{x,x'} W_{\theta} \times \overline{P}^1_{\theta}(x,x') 
[-\frac{d^2}{d\theta^2} \log W_{\theta} (x|x') 
-\frac{d^2}{d\theta^2} \log \overline{P}^1_{\theta}(x')] \nonumber \\
= &
\sum_{x,x'} W_{\theta} \times \overline{P}^1_{\theta}(x,x') 
[-\frac{d^2}{d\theta^2} \log W_{\theta} (x|x') ]
+
\sum_{x'} \overline{P}^1_{\theta}(x') 
[-\frac{d^2}{d\theta^2} \log \overline{P}^1_{\theta}(x')] \nonumber \\
= &
\sum_{x,x'} W_{\theta} \times \overline{P}^1_{\theta}(x,x') 
[-\frac{d^2}{d\theta^2} \log W_{\theta} (x|x') ]
+
J_{\theta}^1\nonumber \\
= &
\sum_{x,x'} W_{\theta} \times \overline{P}^1_{\theta}(x,x') 
\Bigl[-\frac{d^2}{d\theta^2} \log \frac{1}{\lambda_{\theta}}
-\frac{d^2}{d\theta^2} \log \frac{\overline{P}^3_{\theta}(x)}{\overline{P}^3_{\theta}(x')}
-\frac{d^2}{d\theta^2} \log W(x|x')
\nonumber \\
&\hspace{5ex}-\frac{d^2}{d\theta^2} \theta g(x,x')\Bigr]
+
J_{\theta}^1\nonumber \\
\stackrel{(a)}{=} &
\sum_{x,x'} W_{\theta} \times \overline{P}^1_{\theta}(x,x') 
[-\frac{d^2}{d\theta^2} \log \frac{1}{\lambda_{\theta}}]
+
J_{\theta}^1 
= 
\frac{d^2\phi}{d\theta^2} (\theta)
+
J_{\theta}^1, 
\Label{25-1}
\end{align}
where $(a)$ follows from 
the relation $\sum_{x,x'} W_{\theta} \times \overline{P}^1_{\theta}(x,x') 
\frac{d^2}{d\theta^2} \log \frac{\overline{P}^3_{\theta}(x)}{\overline{P}^3_{\theta}(x')}
=0$, which is shown by the following fact:
The expectations of  
$\frac{d^2}{d\theta^2} \log {\overline{P}^3_{\theta}(X)}$ and
$\frac{d^2}{d\theta^2} \log {\overline{P}^3_{\theta}(X')}$ are the same
because the marginal distributions of $X$ and $X'$ are the same.
Hence, we obtain (\ref{27-10}).
The Fisher information $J_{\theta}^2$ is also written as 
\begin{align}
\nonumber
& J_{\theta}^2
= \sum_{x,x'} W_{\theta} \times \overline{P}^1_{\theta}(x,x') 
\Bigl(\frac{d}{d\theta}
\log W_{\theta} (x|x')  \overline{P}^1_{\theta}(x') \Bigr)^2\\
= & \sum_{x,x'} W_{\theta} \times \overline{P}^1_{\theta}(x,x') 
\Biggl[\Bigl(\frac{d}{d\theta}
\log W_{\theta} (x|x')\Bigr)^2\nonumber \\
&+
2 \Bigl(\frac{d}{d\theta}\log W_{\theta} (x|x')\Bigr)
\Bigl(\frac{d}{d\theta} \log \overline{P}^1_{\theta}(x') \Bigr)
+
\Bigl(\frac{d}{d\theta}
\log \overline{P}^1_{\theta}(x') \Bigr)^2\Biggr] \nonumber \\
= &
\sum_{x,x'} W_{\theta} \times \overline{P}^1_{\theta}(x,x') 
\Biggl[ \Bigl( \frac{d}{d\theta}
\log W_{\theta} (x|x')\Bigr)^2 \Biggl] 
+
\sum_{x'} \overline{P}^1_{\theta}(x') 
\Bigl(\frac{d}{d\theta}
\log \overline{P}^1_{\theta}(x') \Bigr)^2 \Biggr] \nonumber \\
&+
2\sum_{x,x'} 
\Bigl(\frac{d}{d\theta}
\log W_{\theta} (x|x') \Bigr) W_{\theta}(x|x') 
\Bigl(\frac{d}{d\theta} \log \overline{P}^1_{\theta}(x') \Bigr) 
\overline{P}^1_{\theta}(x') 
\nonumber \\
= &
\sum_{x,x'} W_{\theta} \times \overline{P}^1_{\theta}(x,x') 
\Bigl[\frac{d}{d\theta}
\log W_{\theta} (x|x')\Bigr]^2
+ 
J_{\theta}^1 \nonumber \\
= &
\sum_{x,x'} W_{\theta} \times \overline{P}^1_{\theta}(x,x') 
\Bigl[- \frac{d \phi}{d \theta}(\theta)
+ \frac{d}{d\theta} \log \overline{P}^3_{\theta}(x)
- \frac{d}{d\theta} \log \overline{P}^3_{\theta}(x')
+ g(x,x')\Bigr]^2
+ 
J_{\theta}^1. 
\Label{25-2}
\end{align}
Combining (\ref{27-10}) and (\ref{25-2}), we have
\begin{align}
\frac{d^2\phi}{d\theta^2} (\theta)
=
\mathsf{V}_\theta [
(g(x,x')- \frac{d \phi}{d \theta}(\theta))
+ \frac{d}{d\theta} \log \overline{P}^3_{\theta}(x)
- \frac{d}{d\theta} \log \overline{P}^3_{\theta}(x')
]
>0,
\end{align}
which implies (\ref{27-11}).
Since
\begin{align*}
&
(\frac{d^2\phi}{d\theta^2} (\theta)
+(\frac{d}{d\theta} \phi(\theta))^2)
e^{\phi(\theta)}
=\frac{d^2}{d\theta^2} e^{\phi(\theta)}
=
\sum_{x,x'}\frac{d^2}{d\theta^2}
W(x|x') e^{\theta g(x,x')} \overline{P}^2_{\theta}(x') \nonumber \\
=&
\sum_{x,x'}
W(x|x') e^{\theta g(x,x')} \overline{P}^2_{\theta}(x')
g(x,x')^2
+2 W(x|x') e^{\theta g(x,x')} 
\frac{d \overline{P}^2_{\theta}(x')}{d \theta}
g(x,x') \nonumber \\
&\quad +
W(x|x') e^{\theta g(x,x')} 
\frac{d^2 \overline{P}^2_{\theta}(x')}{d \theta^2},
\end{align*}
we have another expression of $\frac{d^2\phi}{d\theta^2} (\theta)$
as follows.
\begin{align*}
&
\frac{d^2\phi}{d\theta^2} (\theta)\nonumber \\
=&
e^{-\phi(\theta)}
\Bigl[\sum_{x,x'}
W(x|x') e^{\theta g(x,x')} \overline{P}^2_{\theta}(x')
g(x,x')^2
+2 W(x|x') e^{\theta g(x,x')} 
\frac{d \overline{P}^2_{\theta}(x')}{d \theta}
g(x,x') \nonumber \\
&\quad +
W(x|x') e^{\theta g(x,x')} 
\frac{d^2 \overline{P}^2_{\theta}(x')}{d \theta^2}\Bigr]
-(\frac{d}{d\theta} \phi(\theta))^2.
\end{align*}
When $\theta=0$,
\begin{align*}
\frac{d^2\phi}{d\theta^2} (0)
=&[\sum_{x,x'}
W(x|x') \overline{P}^2_{0}(x')
g(x,x')^2
+2 W(x|x') g(x,x') \frac{d \overline{P}^2_{\theta}(x')}{d \theta}
\Bigr|_{\theta=0}
]
-\eta(0)^2 \nonumber \\
=&
\mathsf{V}_0 [g(X,X')]
+
2 \sum_{x,x'}
W(x|x') g(x,x') \frac{d \overline{P}^2_{\theta}(x')}{d \theta}\Bigr|_{\theta=0}
\end{align*}
because $
\sum_{x,x'}W(x|x') 
\frac{d^2 \overline{P}^2_{\theta}(x')}{d \theta^2}=
\frac{d^2}{d \theta^2}
\sum_{x,x'}W(x|x')  \overline{P}^2_{\theta}(x')=0$
and 
$\eta(0)=\mathsf{E}_0[g(X,X')]$.
Hence, we obtain (\ref{27-12}).

\section{Twice differentiability}
We show the twice-differentiablity of $\phi(\theta)$,
$\overline{P}_\theta^2 $ and $\overline{P}_\theta^3 $.
First, focus on the $1$-parameter case.
Now, we define the function $F_1(\theta,z):= \det (\overline{W}_{\theta}- z I)$ with the identity matrix $I$.
Since $\lambda_\theta=e^{\phi(\theta)}$ is the unique solution of $F_1(\theta,z)=0$
and the function $F_1(\theta,z)$ is twice-differentiable,
the implicit function theorem guarantees that
$\lambda_\theta$ is twice-differentiable.
Hence, $\phi(\theta)$ is also twice-differentiable.

Next, we show that the twice-differentiablity of
$\overline{P}_\theta^2 $ and $\overline{P}_\theta^3 $,
which are normalized eigenvector with positive entries of
$\overline{W}_{\theta}$ and $\overline{W}_{\theta}^T$.
Now, we define the vector-valued function 
$F_2(\theta,y):= \overline{W}_{\theta}y $ 
and
the function $F_3(\theta,y):=\sum_{x\in {\cal X}} y_x$.
Since $\overline{P}_\theta^3 $ 
is the unique solution of $F_2(\theta,y)=0$ and $F_3(\theta,y)=1$
and the functions $F_2(\theta,y)$ and $F_3(\theta,y)$ are twice-differentiable,
the implicit function theorem guarantees that
$\overline{P}_\theta^2$ is twice-differentiable.
Replacing the role of $\overline{W}_{\theta} $ by that of $\overline{W}_{\theta}^T$,
we can show the twice-differentiablity of $\overline{P}_\theta^3 $.
These discussions can be extended to the case when $\theta$ is a $d$-dimensional parameter.

\end{document}